\documentclass[12pt,a4paper,requno]{amsart}
\usepackage{layout}
\usepackage[mathcal]{euscript}
\usepackage{anysize}
\marginsize{22mm}{18mm}{17mm}{30mm}
\usepackage{todonotes}
\usepackage{bm}
\usepackage{amsmath,amssymb,amsthm,amscd,amsfonts,comment}
\usepackage{hyperref}
\usepackage{todonotes}
\usepackage{graphics}
\usepackage{mathtools}

\usepackage{mathrsfs}
\def\Z{\mathbb Z}

\def\R{\mathbb R}
\def\C{\mathbb C}
\def\J{\mathcal J}
\def\I{\mathcal I}
\def\F{F}
\def\H{\mathbb H}

\def\Norm{\mathcal{N}}
\def\L{\mathcal L}

\def\Im{\operatorname{Im}}
\def\Re{\operatorname{Re}}

\def\tr{\operatorname{tr}}

\def\res{\operatorname{res}}

\def\covol{\operatorname{covol}}
\def\vol{\operatorname{vol}}

\def\PSL{\operatorname{PSL}}

\def\log{\operatorname{log}}

\usepackage[toc]{appendix}
\numberwithin{equation}{section}

\theoremstyle{plain}
\newtheorem{theorem}{Theorem}[section]

\newtheorem{lemma}[theorem]{Lemma}
\newtheorem{proposition}[theorem]{Proposition}

\theoremstyle{remark}

\theoremstyle{definition}
\theoremstyle{definition}

\title{A Jensen--Rohrlich type formula for the hyperbolic 3-space}

\author[Herrero]{S. Herrero}
\address{
Sebasti\'an Herrero,
Department of Mathematical Sciences, 
Chalmers University of Technology and University of Gothenburg,
SE-412 96 Gothenburg, 
Sweden.}
\email{sebastian.herrero.m@gmail.com}

\author[Imamo\={g}lu]{\"O. Imamo{\={g}}lu}
\address{
\"Ozlem Imamo{\={g}}lu,
Department of Mathematics, 
ETH Z\"urich,
R\"amistrasse 101, 
CH-8092 Z\"urich, Switzerland.}
\email{ozlem@math.ethz.ch}

\author[von Pippich]{A.-M. von Pippich}
\address{
Anna-Maria von Pippich,
Fachbereich Mathematik, 
Technische Universit\"at Darmstadt,
Schlo{\ss}gartenstr. 7, 
D-64289 Darmstadt, 
Germany.
}
\email{pippich@mathematik.tu-darmstadt.de}

\author[T\'oth]{\'A. T\'oth}
\address{
\'Arp\'ad T\'oth,
Department of Analysis, 
E\"otv\"os Lor\'and University and MTA R\'enyi Int\'ezet Lend\"ulet Automorphic Research Group,
South Building Room 3.207, 
Budapest, 
Hungary.}
\email{toth@cs.elte.hu}

\begin{document}
\setcounter{tocdepth}{1}
\setcounter{section}{0}
\maketitle
\date{}

\begin{abstract}
In this article we give a Jensen--Rohrlich type formula for a certain class of automorphic functions 
on the hyperbolic 3-space for the group $\mathrm{PSL}_2(\mathcal{O}_K)$. 
\end{abstract}


\section{Introduction}
\subsection{Rohrlich's formula}
The classical Jensen's formula is a well-known theorem of complex analysis which characterizes, 
for a meromorphic function $f$ on the unit disc, the value of the integral of 
$\log|f(z)|$ on the unit circle in terms of the zeros and poles of $f$ inside the unit disc.
An important theorem of Rohrlich \cite{ROH} establishes a version of Jensen's formula for modular 
functions $f$ with respect to the full modular group $\mathrm{PSL}_2(\mathbb{Z})$  and expresses 
the integral of $\log|f(z)|$ over a fundamental domain in terms of special values of Dedekind's eta 
function. 

To be more precise, let $\H^2=\{ \tau=x+iy\,|\, x,y\in \R, y>0\}$, $\Gamma=\mathrm{PSL}_2(\Z)$, and
$X=\mathrm{PSL}_2(\Z)\backslash \H^2$. Let $\Gamma_{\tau}$ denote the 
stabilizer subgroup of $\tau$ in $\Gamma$ and let  $\nu(\tau)$ denote its order.
The hyperbolic measure on $X$ is given by $d\mu(\tau)=dxdy/y^2$ and the hyperbolic Laplacian on $X$
is given by
\begin{align*}
\Delta=-y^2\left(\frac{\partial^2}{\partial x^2}+\frac{\partial^2}{\partial y^2}\right).
\end{align*}
The quotient space $X$ has the structure of a hyperbolic Riemann surface of finite hyperbolic 
volume $\vol(X)=\pi/3$, admitting one cusp which we denote by $\infty$.
The field of modular functions on $X$ is given by $\mathbb{C}(j(\tau))$, with $j(\tau)$ denoting Klein's $j$-invariant \cite{Serre} satisfying
\begin{align*}
j(\tau)=\frac{1}{q_{\tau}}+744+O(q_{\tau}),
\end{align*}
as $\tau\to\infty$, where $q_{\tau}=e^{2\pi i \tau}$. 

Consider now the class $\mathcal{M}$ of functions $F:\H^2\to \R\cup\{\infty\}$ satisfying the following properties:
\begin{enumerate}
\item[($\mathcal{M}1$)] The function $F(\tau)$ is $\Gamma$-invariant and can therefore be considered as a function on $X$.
\item[($\mathcal{M}2$)] There exist distinct points $\tau_1,\ldots ,\tau_m\in X$ together with constants $n_1,\ldots ,n_m\in \Z$ satisfying $\sum_{\ell=1}^{m}n_{\ell}=0$ such that, for $\ell\in\{1,\ldots ,m\}$, 
the bound
\begin{equation*}
F(\tau)=n_{\ell}\, \nu(\tau_{\ell}) \log|\tau-\tau_{\ell}|^{-1}+O(1),
\end{equation*}
as $\tau \to \tau_{\ell}$, holds 
and such that $F(\tau)$ is smooth 
at any point $\tau\in X$ with $\tau\not= \tau_{\ell}$ for $\ell\in\{1,\ldots ,m\}$.
\item[($\mathcal{M}3$)]  For $\tau\in X$ with $\tau\not=\tau_{\ell}$ for $\ell\in\{1,\ldots ,m\}$, 
we have $\Delta F(\tau)=0$.
\item[($\mathcal{M}4$)] The function $F(\tau)$ is square-integrable on $X$.
\end{enumerate}
If $F:\H^2\to \mathbb{R}\cup\{\infty\}$ satisfies the properties $(\mathcal{M}1)$--$(\mathcal{M}4)$, 
then the limit $F(\infty):=\lim_{\tau \to \infty}F(\tau)$ exists 
and we have the equality 
\begin{align}\label{expression} 
F(\tau)=\log|f(\tau)|,\quad \mbox{ with } f(\tau)=e^{F(\infty)}\prod_{\ell=1}^m \left(j(\tau)-j(\tau_{\ell})\right)^{-n_{\ell}}.
\end{align}
Now, Rohrlich's Theorem can be rephrased as follows 
\begin{theorem}[Rohrlich \cite{ROH}]\label{thm-rohrlich} 
Let  $F:\H^2\to \R\cup\{\infty\}$ be in $\mathcal{M}$, the class of functions satisfying 
the properties $(\mathcal{M}1)$--$(\mathcal{M}4)$. Then, we have the equality
\begin{align*}
\frac{3}{\pi}\int_{X}F(\tau) d\mu(\tau)=F(\infty)+6\sum_{\ell=1}^{m}
n_{\ell}\log \big(|\eta(\tau_{\ell})|^4\, \mathrm{Im}(\tau_{\ell})\big),
\end{align*}
where 
$\eta(\tau)=q_{\tau}^{1/24}\prod_{n=1}^{\infty}\left(1 - q_{\tau}^{n}\right)$
 is the classical Dedekind's eta function.
\end{theorem} 
Observe that the function on the right hand side of the equality in Theorem \ref{thm-rohrlich}
is given by the constant term in the Laurent expansion of the non-holomorphic Eisenstein series 
$E_{\infty}(\tau,s)$ at $s=1$. For $\tau\in\mathbb{H}^2$ and $s\in\mathbb{C}$ with $\mathrm{Re}(s)>1$, this series is defined by
\begin{align*}
E_{\infty}(\tau,s)=\sum_{\gamma\in \Gamma_{\infty}\backslash \Gamma}\mathrm{Im}(\gamma \tau)^{s}.
\end{align*}
The Eisenstein series is $\Gamma$-invariant with respect to $\tau$ and holomorphic in $s$, and it admits a meromorphic continuation to the whole complex $s$-plane
with a simple pole at $s=1$ with residue 
\begin{align*}
\res_{s=1}E_{\infty}(P,s)=\frac{1}{\vol(X)}=\frac{3}{\pi}.
\end{align*}  
In this context,
the well-known Kronecker's limit formula for $\textrm{PSL}_2(\mathbb{Z})$
(see, e.g., \cite{Siegel80}) states
\begin{align*}
\lim_{s\to1}\left(E_{\infty}(z,s)-
\frac{3}{\pi(s-1)}\right)=
-\frac{3}{\pi}\log\bigl(|\eta(z)|^{4}\Im(z)\bigr)+C,
\end{align*}
where $C=6(1-12\,\zeta'(-1)-\log(4\pi))/\pi$ and $\zeta(s)$ denotes the Riemann zeta function.
Note that the constant $C$ does not appear in Theorem \ref{thm-rohrlich}, since $\sum_{\ell=1}^{m}n_{\ell}=0$.

The proof of Rohrlich's formula is an application of this Kronecker's limit formula.
The formula admits several generalizations and has many applications in number theory, 
see, e.g. \cite{funke}, \cite{kudla}.
There is also an extension of Rohrlich's formula which has applications
to the computation of arithmetic intersection numbers in Arakelov theory, see, e.g., \cite{kuehn}.

\subsection{Purpose of the article}
The goal of this paper is to give an analogue of Rohrlich's formula in $\H^3$, the hyperbolic 3-space. 
We write $\H^3=\{P=z+rj\mid z\in \C,r\in\mathbb{R}_{>0}\}$, which is a subset of the usual quaternions 
$\R[i,j,k]$, and we will view $z$ and $r$ as coordinate functions on $\H^3$. The quaternionic norm on $\R[i,j,k]$ induces a norm on $\H^3$ given 
explicitly by $\|P\|=\sqrt{|z|^2+r^2}$.
We let $K$ be an imaginary quadratic field, $\mathcal{O}_K$ its ring of integers, $h_K$ its class number,
and $d_K$ its discriminant. From now on, we let
$\Gamma=\mathrm{PSL}_2(\mathcal{O}_K)\subset \mathrm{PSL}_2(\mathbb{C})$, which is
a discrete and cofinite subgroup, and we let $X=\Gamma\backslash \H^3$.
By $\Gamma_{P}$ we denote the 
stabilizer subgroup of $P$ in $\Gamma$ and by $\nu(P)$ its order.
By $d\mu(P)$ we denote the hyperbolic measure on $X$
and by $\Delta$ the hyperbolic Laplacian on $X$ (see \ref{hyp}).
The quotient space $X$ has finite hyperbolic volume, which is explicitly given by
\begin{equation}\label{volX}
\vol(X)=\frac{|d_K|^{3/2}}{4\pi^2}\,\zeta_{K}(2)
\end{equation}
with $\zeta_K(s)$ denoting the Dedekind zeta function, and it admits  $h_K$ cusps (see Section \ref{section2}).

For  $P\in \H^3$ and $s\in \C$ with $\Re(s)>1$, the Eisenstein series
associated to the cusp $\infty$ is defined by
\begin{align*}
E_{\infty}(P,s)=
\sum_{\gamma\in \Gamma_{\infty}^{\prime}
\backslash \Gamma}r(\gamma P)^{s+1},
\end{align*}
where $\Gamma_{\infty}^{\prime}$ is the maximal unipotent subgroup of the 
stabilizer group $\Gamma_{\infty}$ of $\infty$ in $\Gamma$. 
The Eisenstein series is $\Gamma$-invariant with respect to $P$ and holomorphic in $s$, 
and it admits a meromorphic continuation to the whole complex $s$-plane
with a simple pole at $s=1$ with residue 
\begin{align}\label{ResEisInfinity}
\res_{s=1}E_{\infty}(P,s) =\frac{\covol(\mathcal{O}_{K})}{\vol(X)}
=\frac{2\pi^2}{|d_K|\zeta_K(2)}.
\end{align}
Here, $\covol(\mathcal{O}_{K})$ denotes the euclidean covolume of the 
lattice $\mathcal{O}_{K}$ in $\C$. In this case, Kronecker's limit formula states 
\begin{align}\label{KLF}
\lim_{s\to 1}\left(E_{\infty}(P,s)- 
\frac{2\pi^2}{|d_K|\zeta_K(2)(s-1)}\right)=-\frac{2\pi^2}{|d_K|\zeta_K(2)} 
\log\left(\eta_{\infty}(P)\, r(P)\right)
+ C_K,
\end{align}
where $C_K$ is an explicit constant depending only on $K$. Here, 
the function $\eta_{\infty}:\H^3\to \R$ satisfies 
$
\eta_{\infty}(\gamma P)=\|cP+d\|^{2}\eta_{\infty}(P)
$
for any $\gamma=\begin{psmallmatrix}a&b\\c&d\end{psmallmatrix}\in\Gamma$
and can be considered as the analogue of the weight 2 real-analytic 
modular form $|\eta(z)|^4$. The function $\eta_{\infty}$ is essentially the function 
defined by Asai in \cite{ASA}. More precisely, we have
\begin{align*}
-\frac{2\pi^2\log\left(\eta_{\infty}(P)\right)}{|d_K|\zeta_K(2)}=
\frac{|\mathcal{O}_K^{\times}|}{2} r^{2}
+
4 \pi
\sum_{\substack{\mu \in \mathcal{D}^{-1}\\ \mu\neq 0}}
|\mu| \varphi_{\infty,\infty}(\mu;1)\, r K_{1}(4\pi |\mu|r)e^{2\pi i \tr(\mu z)}.
\end{align*}
Here, we employed the notation of Section \ref{section2}.
The value $\varphi_{\infty,\infty}(\mu;1)$ can be explicitly given in
terms of special values of certain generalized divisors sums. For these results, 
we refer the reader to \cite{EGM}, Chapter 8, Sections 1--3.

Consider now the class  $\mathcal{A}$ of functions $F:\H^3\to \R\cup\{\infty\}$ satisfying the 
following properties:
\begin{enumerate}
\item[($\mathcal{A}1$)] The function $F(P)$ is $\Gamma$-invariant and can therefore 
be considered
as a function on $X$.
\item[($\mathcal{A}2$)] There exist distinct points $Q_1,\ldots,Q_m\in X$ together 
with constants $c_1,\ldots,c_m\in \R$ satisfying $\sum_{\ell=1}^{m}c_{\ell}=0$ such 
that, for $\ell\in\{1,\ldots,m\}$, the bound
\begin{equation*}
F(P)=c_{\ell}\,\nu(Q_{\ell}) \frac{r_{\ell}}{ \|P-Q_{\ell}\|}+O(1), 
\end{equation*}
as $P \to Q_{\ell}=z_{\ell}+r_{\ell} j$, holds and $F(P)$ is smooth at any point $P\in X$
with $P\not=Q_{\ell}$ for $\ell\in\{1,\ldots ,m\}$.
\item[($\mathcal{A}3$)] For $P\in X$ with $P\not=Q_{\ell}$ for $\ell\in\{1,\ldots ,m\}$, 
we have $\Delta F(P)=0$.
\item[($\mathcal{A}4$)] The function $F(P)$ is square-integrable on $X$.
\end{enumerate}

We note that the bounds in $(\mathcal{M}2)$ and $(\mathcal{A}2)$ are the natural bounds that arise from the type of singularities of the corresponding Green's functions.

In Proposition \ref{thm_block} of Section \ref{section_proofofmainthm}, we 
will show that, if $F:\H^3\to \R\cup\{\infty\}$ satisfies the properties $(\mathcal{A}1)$--$(\mathcal{A}4)$,
then the limit $F(\infty):=\lim_{r \to \infty}F(P)$ exists, and we will 
prove the analogue of \eqref{expression} in this case. Our main theorem is 
\begin{theorem}\label{mainth}
Let $F:\H^3\to \R\cup\{\infty\}$ be in $\mathcal{A}$, the class of functions
satisfying the properties $(\mathcal{A}1)$--$(\mathcal{A}4)$.
Then, we have the equality 
\begin{align*}
\frac{1}{\vol(X)}\int_{X}F(P) d\mu(P)= F(\infty)+\frac{2\pi}{\vol(X)}
\sum_{\ell=1}^{m} c_{\ell} \log\big(\eta_{\infty}(Q_{\ell})\, r_{\ell}\big).
\end{align*}
\end{theorem}
Note that, in analogy with Rohrlich's Theorem, the constant $C_K$ 
arising in \eqref{KLF} does not appear in Theorem \ref{mainth}, since $\sum_{\ell=1}^{m}c_{\ell}=0$.

It is known to the experts that Rohrlich's formula can be proven using 
the theory of the resolvent kernel of the hyperbolic Laplacian and our proof of 
Theorem \ref{mainth} is a generalization of this method to the hyperbolic 3-space.
The advantage of this method is that it can be
generalized to other settings such as the case of the 
hyperbolic $n$-space.
This method also naturally leads to an 
analogue of the function $\log|j(\tau_1)-j(\tau_2)|$
(see the function defined in \eqref{def_analogj}). The properties of this function play a central role in our proof of Theorem \ref{mainth}, and the proof of these follow from properties of the resolvent kernel and of Niebur type Poincar\'e series.

\subsection{Outline of the article}
The paper is organized as follows.
In Section \ref{section2}, we begin by collecting background information. In Section \ref{section-Fourier-expansions}, we compute the Fourier expansion of the resolvent kernel associated to the hyperbolic Laplacian on $X$. In addition, we give the Fourier expansion of the Niebur type Poincar\'e series which appear as coefficients in the Fourier expansion of the resolvent kernel. 
To the best of the authors' knowledge these expansions have not been explicitly stated 
elsewhere in the literature and are of independent interest.
In Section \ref{section-analytic-continuation}, 
we study some of the analytic properties of the Niebur 
type Poincar\'e series and we prove the meromorphic continuation of the resolvent 
kernel via its Fourier expansion.
In Section \ref{section_proofofmainthm}, we construct the above mentioned analogue of $\log|j(\tau_1)-j(\tau_2)|$, prove its main properties, and give our proof of Theorem \ref{mainth} using these properties.
Identities involving special functions that are needed in the paper as well as some technical
lemmas are given in the Appendix and in Section \ref{section-technical-lemma}, respectively. 

\subsection{Acknowledgements}
The authors would like to thank the anonymous referee for helpful comments on an earlier version of this paper. 
Herrero, von Pippich, and T\'oth thank the Institute for Mathematical Research FIM at ETH Z\"urich for providing a stimulating and comfortable atmosphere during their visits to Z\"urich. Herrero, Imamo{\={g}}lu, and von Pippich 
thank J\"urg Kramer and the Department of Mathematics at Humboldt-Universit\"at zu Berlin for their kind hospitality during the preparation of this work. T\'oth thanks the support of the MTA R\'enyi Int\'ezet Lend\"ulet Automorphic Research Group and the NKFIH (National Research, Development and Innovation Office) grant ERC$\underline{\ }$HU$\underline{\ }$15 118946.

\section{Background material}\label{section2}
\subsection{The hyperbolic 3-space and the group $\PSL_2(\mathcal{O}_K)$}
Let
$
\H^3:=\{ P=z+rj\,|\, z\in\C, r\in\R_{>0}\}
$
denote the upper half-space model of the three-dimensional hyperbolic space,
where $\{1,i,j,k\}$ is the standard basis for the quaternions $\R[i,j,k]$. 
The quaternionic norm on $\R[i,j,k]$ 
induces a norm on $\H^3$ given explicitly by $\|P\|=\|z+rj\|=\sqrt{|z|^2+r^2}.$
For $z\in\C$, we set $\tr(z):=z+\overline{z}$. The hyperbolic volume element, resp.~the 
hyperbolic Laplacian are given as
\begin{align}\label{hyp}
d\mu (P):=\frac{dx\,dy\, dr}{r^{3}},\quad\textrm{resp.}
\quad\Delta:=-r^{2}\left(\frac{\partial^{2}}{\partial
x^{2}}+\frac{\partial^{2}}{\partial y^{2}}+\frac{\partial^{2}}{\partial
r^{2}}\right)+r \frac{\partial}{\partial r}\,.
\end{align}
Let $d(P,Q)$ denote the hyperbolic distance between the points $P$ and $Q$.
An explicit formula is given by
\begin{align}\label{formula_cosh}
\cosh\left(d(P,Q)\right)=
\frac{|z_1-z_2|^2+r_1^2+r_2^2}{2r_1r_2},
\end{align}
where $P=z_1+r_1j$ and $Q=z_2+r_2j$.
An element $\gamma=\bigl(\begin{smallmatrix}a&b\\c&d\end{smallmatrix}\bigr)\in \mathrm{PSL}_2(\C)$
acts on $\H^3$ by  
\begin{align*}
\gamma P=
\frac{(az+b)\overline{(cz+d)}+a\overline{ c} r^2}{|cz+d|^2+|c|^2r^2}+ \frac{r}{|cz+d|^2+|c|^2r^2}\,j,
\end{align*}
where $P=z+rj$. By abuse of notation, we represent an element of $\mathrm{PSL}_2(\C)$ by a matrix.

As mentioned in the Introduction, we let $K$ be an imaginary quadratic field, $\mathcal{O}_K$ 
its ring of integers, $h_K$ its class number, and $d_K$ its discriminant. 
We let $\Gamma=\mathrm{PSL}_2(\mathcal{O}_K)\subset \mathrm{PSL}_2(\mathbb{C})$
and we let $X:=\Gamma\backslash \H^3$. By $\Gamma_{P}$ we denote the 
stabilizer subgroup of $P$ in $\Gamma$ and by $\nu(P)$ its order.
In a slight abuse of notation, we will at times identify $X$ with a
fundamental domain in $\H^3$ and identify points on $X$ with their preimages in
such a fundamental domain. 
The hyperbolic volume $\vol(X)$ of $X$ is given by formula \eqref{volX} in terms
of a special value of Dedekind's zeta function, which is defined by
\begin{align*}
\zeta_{K}(s)=\sum_{\substack{I\subseteq \mathcal{O}_K \text{ideal}\\ I \not=(0)}}\Norm(I)^{-s},
\end{align*}
where $s\in\mathbb{C}$ with $\Re(s)>1$ and $\Norm(I)$ denotes the norm of $I$.

A cusp of $X$ is the $\Gamma$-orbit of a parabolic fixed point of $\Gamma$, and
$X$ has $h_K$ cusps. From now on we fix a complete set of representatives 
$C_{\Gamma}\subseteq \mathbb{P}^1(K)$ for the cusps of $X$. 
We write elements of $C_{\Gamma}$ as $[a:b]$ 
for $a,b\in\mathcal{O}_K$, not both equal to $0$, and we write $\infty:=[1:0]$ and
assume that $\infty\in C_{\Gamma}$.
Furthermore, for any cusp $\kappa=[a:b]\in C_{\Gamma}$, we fix a scaling matrix 
$\sigma_{\kappa}=\begin{psmallmatrix}a &\ast \\b& \ast \end{psmallmatrix}\in\mathrm{PSL}_2(K)$ 
such that $\sigma_{\kappa}\infty=\kappa$ and
\begin{align}\label{first-scaling-sigma}
\sigma_{\kappa}^{-1}\Gamma_{\kappa}\sigma_{\kappa}=\left\{ \begin{pmatrix}u& \lambda
\\0&u^{-1}\end{pmatrix}\bigg|\, 
u\in\mathcal{O}_K^{\times}, \lambda \in \Lambda_{\kappa}\right\}
\end{align}
with the full lattice $\Lambda_{\kappa}=(a\mathcal{O}_K+b\mathcal{O}_K)^{-2}\subseteq\mathbb{C}$
(see, e.g., \cite{mazanti}).
 For the cusp $\infty$,
we choose $\sigma_{\infty}$ to be the identity. 
Furthermore, for the 
maximal unipotent subgroup $\Gamma_{\kappa}^{\prime}$, which consists 
of all the parabolic elements of $\Gamma_{\kappa}$ together with the identity,  
we have
\begin{align}\label{scaling-sigma}
\sigma_{\kappa}^{-1}\Gamma^{\prime}_{\kappa}\sigma_{\kappa}=
\left\{ \begin{pmatrix}1&\lambda\\0&1\end{pmatrix}\bigg|\,
\lambda\in \Lambda_{\kappa} \right\}.
\end{align}
We let $\Lambda_{\kappa}^{*}=\{\nu\in \C: \tr(\nu \lambda)\in \Z \text{ for any }\lambda \in \Lambda_{\kappa}\}$ denote its dual lattice.
In particular, we have $\Lambda_{\infty}=\mathcal{O}_{K}$
and $\Lambda_{\infty}^{*}=\mathcal{D}^{-1}$ with
$\mathcal{D}^{-1}=\{\nu\in K\mid\tr(\nu\lambda)\in \Z \text{ for any } \lambda\in \mathcal{O}_K\}$
denoting the inverse different. 

\subsection{Fourier expansion of automorphic functions}\label{subsec_auto_functions}
A function $f:\H^3\to \C$ is called automorphic with respect to $\Gamma$ if it is 
$\Gamma$-invariant, that is, $f(\gamma P)=f(P)$ for any $\gamma\in\Gamma$. An important 
tool to study the behavior of an automorphic function $f$
at a cusp $\xi\in C_{\Gamma}$, with scaling matrix $\sigma_{\xi}$, is its Fourier expansion. 
More precisely, since the function $P\mapsto f(\sigma_{\xi}P)$ is
$\sigma_{\xi}^{-1}\Gamma^{\prime}_{\xi}\sigma_{\xi}$-invariant, employing \eqref{scaling-sigma}, we have
$
f(\sigma_{\xi}(P+\lambda))=f(\sigma_{\xi}P)
$
for any $\lambda\in\Lambda_{\xi}$. 
If $f$ is smooth, the Fourier expansion of 
$f$ with respect to the cusp $\xi$ is therefore of the form
\begin{align}\label{fourier_expansion_general}
f(\sigma_{\xi}P)=\sum_{\mu \in \Lambda_{\xi}^{*}}
a_{\mu}(r)e^{2\pi i \tr(\mu z)},
\end{align}
where $P=z+rj$ and with Fourier coefficients given by
$$a_{\mu}(r)=\frac{1}{\mathrm{covol}(\Lambda_{\xi})}
\int_{\C/\Lambda_{\xi}}f(\sigma_{\xi}P)e^{-2\pi i \tr(\mu z)}dz.$$
If we assume that $f$ is an eigenfunction of the hyperbolic Laplacian, satisfying
$\Delta f=(1-s^2)f$ for some $s\in\mathbb{C}$ with $s\not=0$,
and that $f$ is of polynomial growth
as $r\to \infty$, that is $f(z+rj)=O(r^C)$ as $r\to \infty$ for some constant $C$,
then the expansion \eqref{fourier_expansion_general} has the form
(see, e.g., \cite{EGM}, Theorem 3.1, p.~105)
\begin{align}\label{fourier_expansion_special}
f(\sigma_{\xi}P)=a_0r^{1+s}+b_0 r^{1-s}
+\sum_{\substack{\mu \in \Lambda_{\xi}^{*}\\ \mu\neq 0}}a_{\mu}\,r K_{s}(4\pi |\mu|r)e^{2\pi i \tr(\mu z)}
\end{align}
with $a_0, b_0, a_{\mu}\in\mathbb{C}$ and with $K_{s}(\cdot)$ denoting the modified
Bessel function of the second kind.

\subsection{Poincar\'e series}\label{subsection_poincare}
For later purposes, we define two families of eigenfunctions of the hyperbolic Laplacian, namely the Eisenstein series and the Niebur type Poincar\'e series, which from now one will be called Niebur--Poincar\'e series for simplicity. For this, let $\kappa\in C_{\Gamma}$ be a cusp with scaling matrix $\sigma_{\kappa}$.

For  $P\in \H^3$ and $s\in \C$ with $\Re(s)>1$, the Eisenstein series
associated to the cusp $\kappa$ is given by
\begin{align*}
E_{\kappa}(P,s)=
\sum_{\gamma\in \Gamma_{\kappa}^{\prime}\backslash \Gamma}
r(\sigma_{\kappa}^{-1}\gamma P)^{s+1}.
\end{align*}
The Eisenstein series is an automorphic function for $\Gamma$ and it is holomorphic 
in $s$ in the region $\mathrm{Re}(s)>1$. Moreover, it satisfies the differential 
equation
\begin{align*}
\left(\Delta-(1-s^2)\right)E_{\kappa}(P,s)=0,
\end{align*}
i.e.~it is an eigenfunction of $\Delta$. For $s\in\mathbb{C}$ with $\mathrm{Re}(s)>1$,
the Eisenstein series admits a Fourier expansion of the form \eqref{fourier_expansion_special}
given by (see, e.g., \cite{EGM}, Theorem 4.1, p.~111)
\begin{align}\label{equ_fourier_expansion_series_eis_gen}
E_{\kappa}(\sigma_{\xi}P,s)&=
\delta_{\kappa,\xi} [\Gamma_{\kappa}:\Gamma_{\kappa}^{\prime}]\,
r^{1+s}
+
\varphi_{\kappa,\xi}(0;s)\,r^{1-s}
+
\frac{2^{1+s}\pi^{s}}{\Gamma(s)}
\sum_{\substack{\mu \in \Lambda_{\xi}^{*}\\ \mu\neq 0}}|\mu|^{s}\varphi_{\kappa,\xi}(\mu;s)\,r K_{s}(4\pi |\mu|r)e^{2\pi i \tr(\mu z)},
\end{align}
where $\delta_{\kappa,\xi}$ is Kronecker's delta symbol and, for $s\in\mathbb{C}$ with $\mathrm{Re}(s)>1$, we have set
\begin{align}\label{def_varphiscattering}
\varphi_{\kappa,\xi}(\mu;s):=\frac{\pi}{\covol(\Lambda_{\xi})s}
\sum_{\begin{psmallmatrix}*&*\\c&d \end{psmallmatrix}\in \sigma_{\kappa}^{-1}\mathcal{R}_{\kappa,\xi}\sigma_{\xi}}
\frac{e^{2\pi i \tr(\mu \frac{d}{c})}}{|c|^{2s+2}}
\end{align}
with 
\begin{align}\label{def_R}
\mathcal{R}_{\kappa,\xi}:=\Gamma_{\kappa}^{\prime}\backslash \{\gamma\in \Gamma: \gamma \xi\neq \kappa\}/\Gamma_{\xi}^{\prime}.
\end{align} 
Note that $\{\gamma\in \Gamma: \gamma \xi\neq \kappa\}=\Gamma$ if
$\xi\not=\kappa$. It is known (see, e.g., \cite{EGM}, \cite{SAR}) that the function $\varphi_{\kappa,\xi}(\mu;s)$
admits a meromorphic continuation to all $s\in\mathbb{C}$, which is holomorphic at $s=1$ if $\mu \neq 0$.
It is also well-known that one can use the above Fourier expansion in order to prove that 
$E_{\kappa}(P,s)$ admits a meromorphic continuation to the whole complex $s$-plane. 
There is always a simple pole at $s=1$ with residue given by  
\begin{equation}\label{resEis}
\res_{s=1}E_{\kappa}(P,s)
=\frac{\covol(\Lambda_{\kappa})}{\vol(X)}
=\res_{s=1}\varphi_{\kappa,\kappa}(0;s).
\end{equation}
In case that $\kappa=\infty$, the residue is explicitly given by \eqref{ResEisInfinity} and
we have
\begin{equation*}
\varphi_{\infty,\infty}(0;s)=\frac{\pi |\mathcal{O}_K^{\times}|}{h_K|d_K|^{1/2}s}\sum_{\chi}\frac{L(s,\chi)}{L(s+1,\chi)},
\end{equation*}
where the sum runs over all characters $\chi$ of the class group of $K$ and
$L(s,\chi)$ denotes the associated $L$-function (see, e.g., \cite{EGM}, Chapter 8, Theorems 1.5 and 2.11). From these, a straight-forward computation yields 
the Kronecker's limit formula \eqref{KLF} stated in the Introduction.

Finally, we recall the definition of the Niebur--Poincar\'e series.
For  $P\in \H^3$ and $s\in \C$ with $\Re(s)>1$, the Niebur--Poincar\'e series
associated to the cusp $\kappa$ and to $\nu\in \Lambda_{\kappa}^{*}$, 
$\nu \neq 0$, is given by
\begin{align}\label{def-niebur-series}
F_{\kappa,\nu}(P,s)=
\sum_{\gamma \in \Gamma_{\kappa}'\backslash \Gamma}
r(\sigma_{\kappa}^{-1}\gamma P)\,
I_{s}\left(4\pi |\nu | r(\sigma_{\kappa}^{-1} \gamma P)\right)
e^{ 2\pi i \tr\left(\nu  z(\sigma_{\kappa}^{-1}\gamma P)\right)},
\end{align} 
where $I_s(\cdot)$ denotes the modified Bessel function of the first kind.
We recall that 
the Niebur--Poincar\'e series converges absolutely and defines an automorphic function, 
which is holomorphic for $s\in\mathbb{C}$
with $\mathrm{Re}(s)>1$ (see, e.g., \cite{matthes}).
Moreover, it satisfies the differential equation
\begin{align*}
\left(\Delta-(1-s^2)\right)F_{\kappa,\nu}(\cdot,s)=0,
\end{align*}
i.e.~it is an eigenfunction of $\Delta$.

\subsection{The resolvent kernel}\label{Green-subsection}
The resolvent kernel for the hyperbolic Laplacian is given by the automorphic Green's function.
For $P,Q\in \mathbb{H}^3$ with $P\not=\gamma Q$ for any $\gamma\in\Gamma$,
and $s\in \C$ with $\Re(s)>1$, it is defined by 
\begin{align*}
G_s(P,Q)=\frac{1}{2\pi}\sum_{\gamma\in \Gamma}
\varphi_s\left(\cosh (d(P,\gamma Q))\right),
\end{align*}
where $\varphi_s(t)=(t+\sqrt{t^2-1})^{-s}(t^2-1)^{-1/2}$. The series defining $G_s(P,Q)$ converges uniformly on compact subsets of $\{(P,Q)\in \H^3\times \H^3:P\neq \gamma Q \text{ for any }\gamma \in \Gamma)\}\times \{s\in \C: \mathrm{Re}(s)>1\}$. 
We recall the following well-known properties of $G_s(P,Q)$ (see, e.g., \cite{EGM}):
\begin{enumerate}
\item[(G1)] 
The function $G_s(P,Q)$ is $\Gamma$-invariant in each variable
and can therefore be considered as a function on $X\times X$, away from the diagonal. 
Moreover, we have $G_s(P,Q)=G_s(Q,P)$.
\item[(G2)] 
For fixed $Q\in X$, 
we have a singularity of the form
\begin{equation*}
G_s(P,Q)= \frac{\nu(Q)}{2\pi} \frac{1}{ d(P,Q)}+O_Q(1), 
\end{equation*}
as $P \to Q$.  
\item[(G3)] For $P,Q\in X$ with $P\neq Q$, we have $(\Delta_P-(1-s^2)) G_s(P,Q)=0.$
\end{enumerate}
The Green's function is holomorphic for  $s\in \C$ with $\Re(s)>1$ and it 
admits a meromorphic continuation to the whole complex $s$-plane
with a simple pole at $s=1$ with residue 
\begin{equation}\label{Greenpole}
\res_{s=1}\,G_s(P,Q)=\frac{1}{\vol(X)}.
\end{equation}
Moreover, using the spectral expansion of $G_s(P,Q)$ given in \cite{EGM} (Proposition 4.6, p.~285), 
it is easy to see that the function
\begin{align*}\label{prop_greens}
P  \mapsto \lim_{s\to 1}\left(G_s(P,Q)-\frac{2}{\vol(X)(s^2-1)}\right)
\end{align*}
is square-integrable on $X$, for fixed $Q\in X$, and orthogonal to the constant functions, i.e.
\begin{equation}\label{G-orthogonal}
\int_X\lim_{s\to 1}\left(G_s(P,Q)-\frac{2}{\vol(X)(s^2-1)}\right)d\mu(P)=0.
\end{equation}

\section{Fourier expansions}\label{section-Fourier-expansions}
In this section, we compute the Fourier expansion of the Green's function and that of 
the Niebur--Poincar\'e series. Part of the computations involve explicit evaluations of 
certain integrals in terms of special functions. The proof of these technical identities 
is postponed to Section \ref{section-technical-lemma} in order to keep the exposition simple.

\begin{proposition}\label{prop-fourier-expansion-greens}
Let $P=z+rj\in \H^3$ with $r>r(\sigma^{-1}_{\xi}\gamma Q)$ for any $\gamma\in\Gamma$, 
and $s\in\mathbb{C}$ with $\Re(s)>1$. Then, we have the following Fourier expansion
\begin{equation*}
G_s(\sigma_{\xi}P,Q)=\frac{1}{\covol(\Lambda_{\xi})}\biggl(
 \frac{r^{1-s}}{s}E_{\xi}(Q,s)+2\sum_{\substack{\mu \in \Lambda_{\xi}^{*}\\ \mu\neq 0}} F_{\xi,-\mu}(Q,s)   \, r K_{s}(4\pi  |\mu| r)  e^{2\pi i \tr (\mu z )}\biggr).
\end{equation*}
\end{proposition}

\begin{proof}
The Fourier coefficient $a_{\mu}(r)=a_{\mu,s}(r,Q)$ in the Fourier expansion 
\eqref{fourier_expansion_general} of the function $P\mapsto G_s(P,Q)$ with respect to the cusp $\xi$ is given 
by 
$$a_{\mu,s}(r,Q)=\frac{1}{\covol(\Lambda_{\xi})}\int_{\C/\Lambda_{\xi}}
G_s(\sigma_{\xi}P,Q)e^{-2\pi i \tr ( \mu z )}dz.
$$
To compute this integral, we start by writing
\begin{align*}
G_s(\sigma_{\xi}P,Q)
&=
\frac{1}{2\pi}\sum_{\gamma\in \Gamma_{\xi}'\backslash \Gamma}
\sum_{\eta\in \Gamma_{\xi}'}
\varphi_s\left(\cosh (d(\eta^{-1}\sigma_{\xi}P,\gamma Q))\right)\\
&
=\frac{1}{2\pi}\sum_{\gamma\in \Gamma_{\xi}'\backslash \Gamma}
\sum_{\lambda\in \Lambda_{\xi}}
\varphi_s\left(\cosh (d(P+\lambda,\sigma_{\xi}^{-1}\gamma Q))\right),
\end{align*}
where for the last equality we employed \eqref{scaling-sigma}, 
namely the identity
$
\sigma_{\xi}^{-1}\Gamma^{\prime}_{\xi}\sigma_{\xi}=
\left\{\begin{psmallmatrix}1&\lambda\\0&1\end{psmallmatrix}|\,
\lambda\in \Lambda_{\xi} \right\}
$.
Hence, we get
\begin{align*}
a_{\mu,s}(r,Q)=\frac{1}{2\pi\covol(\Lambda_{\xi})}
\sum_{\gamma\in \Gamma_{\xi}'\backslash \Gamma}\,\int_{\C}
\varphi_s\left(\cosh (d(P,\sigma_{\xi}^{-1}\gamma Q))\right)e^{-2\pi i \tr ( \mu z)}dz.
\end{align*}
Now, we set $\tilde{z}:=z(\sigma_{\xi}^{-1}\gamma Q)$ and $\tilde{r}:=
r(\sigma_{\xi}^{-1}\gamma Q)$.
Using formula \eqref{formula_cosh}, namely
$$
\cosh (d(P,\sigma_{\xi}^{-1}\gamma Q))=
\frac{|z-\tilde{z}|^2+r^2+\tilde{r}^2}{2r\tilde{r}},
$$
we obtain by a change of variables
($z\mapsto z+\tilde{z}$), 
$$
a_{\mu,s}(r,Q)
=\frac{1}{2\pi\covol(\Lambda_{\xi})}
\sum_{\gamma\in \Gamma_{\xi}'\backslash \Gamma}
e^{-2\pi i \tr (\mu  \tilde{z} )}
I_{\mu,s}(r,\tilde{r}),
$$
where we have set
\begin{align*}
I_{\mu,s}(r,\tilde{r}):=\int_{\C}\varphi_s\left(\frac{|z|^2+r^2+\tilde{r}^2}{2r\tilde{r}}\right)e^{-2\pi i \tr (\mu z )}dz.
\end{align*}
By Lemma \ref{Iintegral}, we have
\begin{align*}
I_{\mu,s}(r,\tilde{r})=
\begin{cases}
2\pi s^{-1}r^{1-s}\tilde{r}^{s+1}, &\text{ if }\mu= 0,\\
4\pi r\tilde{r} K_{s}(4\pi |\nu|r) I_{s}(4\pi |\nu|\tilde{r}),  &\text{ if }\mu \neq 0.
\end{cases}
\end{align*}
Summing up and recalling that $\tilde{r}=r(\sigma_{\xi}^{-1}\gamma Q)$, we conclude
\begin{align*}
a_{0,s}(r,Q)
=\frac{1}{\covol(\Lambda_{\xi})} \frac{r^{1-s}}{s}
\sum_{\gamma\in \Gamma_{\xi}'\backslash \Gamma}
r(\sigma_{\xi}^{-1}\gamma Q)^{s+1}=\frac{1}{\covol(\Lambda_{\xi})} \frac{r^{1-s}}{s}E_{\xi}(Q,s)
\end{align*}
and, for $\mu\neq 0$, we derive
\begin{align*}
a_{\mu,s}(r,Q)
&=\frac{2}{\covol(\Lambda_{\xi})} r K_{s}(4\pi  |\mu| r) 
\sum_{\gamma\in \Gamma_{\xi}'\backslash \Gamma}
\tilde{r}I_{s}(4\pi  |\mu| \tilde{r}) e^{-2\pi i \tr (\mu  \tilde{z} )}\\
&=\frac{2}{\covol(\Lambda_{\xi})} r K_{s}(4\pi  |\mu| r)
F_{\xi,-\mu}(Q,s),
\end{align*}
as asserted. This completes the proof.
\end{proof}
We proceed by computing the Fourier expansion of the Niebur--Poincar\'e series
$F_{\kappa,\nu}(P,s)$, where $\kappa\in C_{\Gamma}$ and $\nu\in \Lambda_{\kappa}^{*}$, $\nu \neq 0$.
For this, we define the function $\J_s:\C^{\times}\to \C$ by 
\begin{align}\label{def_Lniebur}
\J_s(z):=
\begin{cases}
J_s(4\pi \sqrt{z})J_s(4\pi \sqrt{\overline{z}}), &\text{ if } \Re(z)\geq 0,\\
I_s(4\pi \sqrt{-z})I_s(4\pi \sqrt{-\overline{z}}), &  \text{ if } \Re(z)\leq 0.  
\end{cases}
\end{align}
Using the identity $I_s(z)=e^{\mp s \pi /2}J_s(ze^{\pm \pi i /2})$ for $z\in \C$ with $\Re(z)>0$,
it is easy to verify that $\J_s(z)$ is well-defined for $z\in \C$, $z\neq 0$, with $\Re(z)=0$.
With this, we have

\begin{proposition}\label{FENgeneral}
Let $P=z+rj\in \H^3$ and $s\in\mathbb{C}$ with $\Re(s)>1$. Then, 
we have the following Fourier expansion
\begin{align*}
F_{\kappa,\nu}(\sigma_{\xi} P,s)
&= \delta_{\kappa,\xi}\, rI_{s}(4\pi  |\nu| r)
\sum_{ u\in \mathcal{O}_{K}^{\times}/\{\pm1\}}
e^{2\pi i \tr( \nu u^2z )} +\frac{\mathrm{covol}(\Lambda_{\kappa})}{\mathrm{covol}(\Lambda_{\xi})}
 \frac{(2\pi|\nu|)^s}{s\Gamma(s)}\varphi_{\xi,\kappa}(-\nu;s)\,r^{1-s}\\ 
&\mathrel{\phantom{=}} + \sum_{\substack{\mu \in \Lambda_{\xi}^{\ast}\\ \mu\neq 0}}
\mathcal{B}_{\kappa,\xi}(\nu,\mu;s)\,rK_s(4\pi |\mu|r) e^{2\pi i \tr(\mu z)}.
\end{align*}
Here, 
\begin{align*}
\mathcal{B}_{\kappa,\xi}(\nu,\mu;s)
:=\frac{2\pi}{\covol(\Lambda_{\xi})}\,
\sum_{\begin{psmallmatrix}a&*\\c&d \end{psmallmatrix} \in \sigma_{\kappa}^{-1}\mathcal{R}_{\kappa,\xi}\sigma_{\xi}}\frac{e^{2\pi i \tr\left( (\nu a+\mu d)/c\right)}}{|c|^{2}}
\,\J_s\left(\frac{\nu \mu}{c^2}\right),
\end{align*}
and $\varphi_{\xi,\kappa}(-\nu;s)$, $\mathcal{R}_{\kappa,\xi}$, and $\J_s(\cdot)$ are given by \eqref{def_varphiscattering}, \eqref{def_R}, \eqref{def_Lniebur}, respectively.
\end{proposition}

\begin{proof}
To simplify the notation, we set $f(P):=r(P)I_{s}\left(4\pi |\nu | r(P)\right)e^{ 2\pi i \tr\left(\nu  z(P)\right)}$ 
and we define
\begin{align*}
\widehat{F}_{\kappa,\nu}(P,s):= 
 \sum_{\gamma \in \mathcal{R}_{\kappa,\xi}}\sum_{\eta\in\Gamma'_{\xi} }
f(\sigma_{\kappa}^{-1} \gamma \eta P).
\end{align*}
Recalling the definition \eqref{def-niebur-series} of the Niebur-Poincar\'e series, we 
then deduce
\begin{align}
F_{\kappa,\nu}(\sigma_{\xi} P,s) 
&=\sum_{\gamma \in \Gamma_{\kappa}'\backslash \Gamma}
f\left(\sigma_{\kappa}^{-1}\gamma \sigma_{\xi}P\right)\notag\\
&=\delta_{\kappa,\xi}\sum_{\gamma \in \Gamma_{\kappa}'\backslash \Gamma_{\kappa}}
f\left(\sigma_{\kappa}^{-1}\gamma \sigma_{\xi}P\right)
+ \widehat{F}_{\kappa,\nu}(\sigma_{\xi}P,s).\label{splitting}
\end{align}
To treat the first term in \eqref{splitting}, we assume that 
$\delta_{\kappa,\xi}=1$, that is $\kappa=\xi$ and $\sigma_{\kappa}= \sigma_{\xi}$.
Then
\begin{align*}
\delta_{\kappa,\xi}\sum_{\gamma \in \Gamma_{\kappa}'\backslash \Gamma_{\kappa}}
f\left(\sigma_{\kappa}^{-1}\gamma \sigma_{\xi}P\right)
=\sum_{\gamma \in \sigma_{\kappa}^{-1}(\Gamma_{\kappa}'\backslash \Gamma_{\kappa})\sigma_{\kappa}}
f\left(\gamma P\right)
= rI_{s}(4\pi  |\nu| r)
\sum_{ u\in \mathcal{O}_{K}^{\times}/\{\pm1\}}
e^{2\pi i \tr( \nu u^2z )},
\end{align*}
where for the second equality we note that
$\sigma_{\kappa}^{-1}(\Gamma_{\kappa}'\backslash \Gamma_{\kappa})\sigma_{\kappa}\cong  
\left\{ \begin{psmallmatrix}u&0
\\0&u^{-1}\end{psmallmatrix}\mid
u\in\mathcal{O}_K^{\times}\right\}/\{\pm1\}$, which is an immediate consequence of 
\eqref{first-scaling-sigma} and \eqref{scaling-sigma}, and we
used
the identity $\begin{psmallmatrix}u&0\\0&u^{-1} \end{psmallmatrix} P=
u^2z+rj$ for $u\in \mathcal{O}_{K}^{\times}$.

To treat the second term in \eqref{splitting}, we note that 
the function $\widehat{F}_{\kappa,\nu}(\sigma_{\xi} \cdot ,s)$ is $\sigma_{\xi}^{-1}\Gamma'_{\xi} \sigma_{\xi}$-invariant. 
The Fourier coefficient $b_{\mu}(r)=b_{\mu,\kappa,\nu}(r,s)$ in the Fourier expansion 
\eqref{fourier_expansion_general} of the function $\widehat{F}_{\kappa,\nu}(P,s)$ with respect to the cusp $\xi$ is given 
by 
$$b_{\mu,\kappa,\nu}(r,s)
=\frac{1}{\covol(\Lambda_{\xi})}\int_{\C/\Lambda_{\xi}}
\widehat{F}_{\kappa,\nu}(\sigma_{\xi} P,s)e^{-2\pi i \tr ( \mu z )}dz.
$$
To compute this integral, we start by writing
\begin{align*}
\widehat{F}_{\kappa,\nu}(\sigma_{\xi} P,s)
&=
 \sum_{\gamma \in \mathcal{R}_{\kappa,\xi}}\sum_{\eta\in\Gamma'_{\xi} }
f\left(\sigma_{\kappa}^{-1} \gamma \eta \sigma_{\xi} P\right)\\
&=\sum_{\gamma\in \mathcal{R}_{\kappa,\xi}}
\sum_{\lambda\in \Lambda_{\xi}}
f\left(\sigma_{\kappa}^{-1} \gamma \sigma_{\xi}(P+\lambda)\right),
\end{align*}
where for the last equality we employed \eqref{scaling-sigma}, 
namely the identity
$
\sigma_{\xi}^{-1}\Gamma^{\prime}_{\xi}\sigma_{\xi}=
\left\{\begin{psmallmatrix}1&\lambda\\0&1\end{psmallmatrix}|\,
\lambda\in \Lambda_{\xi} \right\}
$.
Hence, we get
\begin{align*}
b_{\mu,\kappa,\nu}(r,s)&=\frac{1}{\covol(\Lambda_{\xi})}
 \sum_{\gamma \in\sigma_{\kappa}^{-1}  \mathcal{R}_{\kappa,\xi}\sigma_{\xi}}\,\int_{\C}
f\left(\gamma P\right)e^{-2\pi i \tr ( \mu z)}dz.
\end{align*}
Now, writing
$
z(\gamma P)=\tfrac{a}{c}-\tfrac{1}{c}  \tfrac{\overline{cz+d}}{|cz+d|^2+|c|^2r^2}
$ with  $\gamma=\bigl(\begin{smallmatrix}a&b\\c&d\end{smallmatrix}\bigr)$ and using 
\begin{align*}
f\left(\gamma P\right)
&=\frac{r}{|cz+d|^2+|c|^2r^2}\,I_s\left(\frac{4\pi |\nu|r}{|cz+d|^2+|c|^2r^2}\right)
e^{2\pi i \tr( \nu\frac{a}{c})}
e^{-2\pi i \tr \left(\tfrac{\nu}{c} \tfrac{\overline{cz+d}}{|cz+d|^2+|c|^2r^2}\right)},
\end{align*}
we obtain
by a change of variables ($z\mapsto z-\tfrac{d}{c}$) 
\begin{align*}
b_{\mu,\kappa,\nu}(r,s)&=\frac{1}{\mathrm{covol}(\Lambda_{\xi})}\sum_{\begin{psmallmatrix}a&*\\c&d \end{psmallmatrix} \in \sigma_{\kappa}^{-1}\mathcal{R}_{\kappa,\xi}\sigma_{\xi}}
e^{2\pi i \tr \left(\nu\frac{a}{c}+ \mu \frac{d}{c}\right)}
\I(r,\nu,\mu,c),   
\end{align*}
where we have set
\begin{align*}
\I(r,\nu,\mu,c):=\int_{\C}\frac{r}{|c|^2(|z|^2+r^2)}
I_s\left(\frac{4\pi |\nu|r}{|c|^2(|z|^2+r^2)}\right)
e^{-2\pi i \tr\left(\nu\tfrac{\overline{z}}{c^2(|z|^2+r^2)}+\mu z\right)}dz.
\end{align*}
By Lemma \ref{Gintegral}, 
we have
\begin{align*}
\I(r,\nu,\mu,c)=
\begin{cases}
\dfrac{2^s\pi^{s+1}|\nu|^s}{|c|^{2(s+1)}s^2\Gamma(s)}r^{1-s}, & \text{ if }\mu =0,\\[0.5cm]
\dfrac{2\pi}{|c|^2}\J_s\left(\frac{\nu \mu}{c^2}\right) rK_s(4\pi |\mu |r), & \text{ if }\mu \neq 0.
\end{cases}
\end{align*}
Summing up, we conclude 
\begin{align*}
b_{0,\kappa,\nu}(r,s)&=\frac{1}{\mathrm{covol}(\Lambda_{\xi})}
\frac{2^s\pi^{s+1}|\nu|^s}{s^2\Gamma(s)}r^{1-s}
\sum_{\begin{psmallmatrix}a&*\\c&* \end{psmallmatrix} \in \sigma_{\kappa}^{-1}\mathcal{R}_{\kappa,\xi}\sigma_{\xi}}\frac{e^{2\pi i \tr( \nu \frac{a}{c})}}{|c|^{2s+2}}\\
&=\frac{1}{\mathrm{covol}(\Lambda_{\xi})}
\frac{2^s\pi^{s+1}|\nu|^s}{s^2\Gamma(s)}r^{1-s}
\sum_{\begin{psmallmatrix}*&*\\c&d \end{psmallmatrix} \in \sigma_{\xi}^{-1}\mathcal{R}_{\xi,\kappa}\sigma_{\kappa}}\frac{e^{2\pi i \tr( -\nu \frac{d}{c})}}{|c|^{2s+2}}.
\end{align*}
Recalling definition \eqref{def_varphiscattering}, we get that
\begin{align*}
b_{0,\kappa,\nu}(r,s)
&=\frac{\mathrm{covol}(\Lambda_{\kappa})}{\mathrm{covol}(\Lambda_{\xi})}
\frac{2^s\pi^{s}|\nu|^s}{s\Gamma(s)}r^{1-s}
\varphi_{\xi,\kappa}(-\nu;s),
\end{align*}
as asserted. Furthermore, for $\mu\not=0$, we conclude
\begin{align*}
b_{\mu,\kappa,\nu}(r,s)&=\frac{2\pi}{\mathrm{covol}(\Lambda_{\xi})}
rK_s(4\pi |\mu |r)
\sum_{\begin{psmallmatrix}a&*\\c&d \end{psmallmatrix} \in \sigma_{\kappa}^{-1}\mathcal{R}_{\kappa,\xi}\sigma_{\xi}}
\frac{e^{2\pi i \tr(\nu\frac{a}{c}+ \mu \frac{d}{c})}}{|c|^2}\J_s\left(\frac{\nu \mu}{c^2}\right) \\
&= \mathcal{B}_{\kappa,\xi}(\nu,\mu;s)rK_s(4\pi |\mu |r).
\end{align*}
This completes the proof.
\end{proof}

\section{Analytic continuation}\label{section-analytic-continuation}
The main goal of this section is to prove the meromorphic continuation of the
Green's function via its Fourier expansion. We remark here that the existence 
of this meromorphic continuation is well-known and follows from the spectral 
expansion of the Green's function (see, e.g., \cite{EGM}, Proposition 4.6, p.~285). 
Here we choose a different approach as we also need precise information about the 
growth at the cusp $\infty$ of this meromorphic continuation. In order to do this, 
we first analytically continue the Niebur--Poincar\'e series $F_{\infty,\nu}(P,s)$ 
by using the explicit Fourier expansion given in Proposition \ref{FENgeneral} 
with $\xi=\infty$. Before doing so, we need the following result.

\begin{lemma}\label{Lbound}
We have the bounds
\begin{align*}
|\J_s(z)|=
\begin{cases}
O\left(| z|^{\Re(s)}\right), & \text{ for }0<|z|\leq 1,\\[0.1cm]
O\left(e^{8\pi \sqrt{|z|}}|z|^{\Re(s)}\right), & \text{ for } |z|> 1,
\end{cases}
\end{align*}
holding uniformly for $s$ in any compact set contained in $\mathrm{Re}(s)>-1/2$.
\end{lemma}
\begin{proof}
Using the asymptotic formula \eqref{IJasympt}, we conclude
$$\J_s(z)\sim \frac{|4\pi^2 z|^s}{\Gamma(s+1)^2},$$
for $z\to 0$.
This implies the first bound. In order to obtain the second bound, we use \eqref{Iinfinity} 
and get
$$
I_s(4\pi \sqrt{-z})I_s(4\pi \sqrt{-\overline{z}})=O\left(|z|^{\Re(s)}
e^{8\pi \Re(\sqrt{-z})
}
\right),
$$
for $z\to \infty,\Re(z)\leq 0$. On the other hand, formula \eqref{Jinfinity}
gives 
$$
J_s(4\pi \sqrt{z})J_s(4\pi \sqrt{\overline{z}})=O\left(|z|^{\Re(s)}
e^{8\pi |\Im(\sqrt{z})|}
\right),
$$
for $z\to \infty, \Re(z)\geq 0$. The second bound follows easily from these estimates. Since the used asymptotic formulas and bounds are uniform for $s$ in any compact set contained in $\Re(s)>-1/2$, we conclude that these estimates are also uniform. This completes the proof of the Lemma.
\end{proof}

Given $\nu,\mu\in \mathcal{D}^{-1}$ both non zero and $s\in \C$ define 
$$
\mathcal{Z}(\nu,\mu;s)
:=\sum_{\substack{c\in \mathcal{O}_K/\{\pm 1\}\\ c\neq 0}} 
\frac{|\mathcal{S}(\nu,\mu,c)|}{|c|^{2+2s}},
$$
where
$$\mathcal{S}(\nu,\mu,c):=\sum_{\substack{u,u^{\ast}\in \mathcal{O}_K/c\mathcal{O}_K\\ uu^{\ast}=1}}e^{2\pi i \tr((u\nu +u^{\ast}\mu)/c)}.$$
By using the trivial bound for $|\mathcal{S}(\nu,\mu,c)|$, namely $|\mathcal{S}(\nu,\mu,c)|\leq \Norm(c)=|c|^2$, one sees that the series $\mathcal{Z}(\nu,\mu;s)$ converges absolutely for $\mathrm{Re}(s)>1$.

\begin{lemma}\label{contLZseries}
The series $\mathcal{Z}(\nu,\mu;s)$ converges absolutely for $\Re(s)>1/2$. Moreover, there exists $\alpha>0$ such that the bound 
$$|\mathcal{Z}(\nu,\mu;s)|=O\left(\Norm(\nu \mu \mathcal{D}^2)^{\alpha}\right)$$
holds uniformly for $s$ in any compact set contained in $\Re(s)> 1/2$.
\end{lemma}
\begin{proof}
This result is essentially due to Sarnak. Indeed, from the proof of Proposition 3.4 in \cite{SAR} 
we have 
$$\sum_{\substack{c\in \mathcal{O}_K/\{\pm 1\}\\ c\neq 0}}  \frac{|\mathcal{S}(\nu,\mu,c)|}{|c|^{2+2\sigma}} 
\leq
\frac{|\mathcal{O}_K^{\times}|}{2}
\prod_{\substack{P\subset \mathcal{O}_K\\  \nu\mu\mathcal{D}^2\subseteq P}}\left(1-\Norm(P)^{-\sigma}\right)^{-1} \prod_{\substack{P\subset \mathcal{O}_K\\ P \neq (0)}}\left(1+2\Norm(P)^{-\frac{1}{2}-\sigma}+\frac{\Norm(P)^{-2\sigma}}{1-\Norm(P)^{-\sigma}}\right),
$$
where $\sigma=\Re(s)$ and the products run over prime ideals $P\subset \mathcal{O}_K$. The infinite product
$$\prod_{\substack{P\subset \mathcal{O}_K\\ P \neq (0)}}\left(1+2\Norm(P)^{-\frac{1}{2}-\sigma}+\frac{\Norm(P)^{-2\sigma}}{1-\Norm(P)^{-\sigma}}\right) $$
converges for $\sigma>1/2$, proving the absolute convergence of $\mathcal{Z}(\nu,\mu,s)$ for $\Re(s)>1/2$. On the other hand, since the function $x\mapsto (1-x^{-\sigma})^{-1}$ is decreasing for $x>1$ and $\Norm(P)\geq 2$ for any prime ideal $P$, we have
$$\prod_{\substack{P\subset \mathcal{O}_K\\  \nu\mu\mathcal{D}^2\subseteq P}}\left(1-\Norm(P)^{-\sigma}\right)^{-1}\leq (1-2^{-\sigma})^{-\ell},$$
where  $\ell$ is the number of prime ideals dividing $\nu\mu \mathcal{D}^2$. But $\ell\leq 2\omega^{\#}(\Norm(\nu \mu \mathcal{D}^2))$, where $\omega^{\#}(n)$ is the number of prime divisors of $n\in \mathbb{N}$. It is known that $\omega^{\#}(n)=O(\log(n))$, which gives $(1-2^{-\sigma})^{-\ell} \leq \Norm(\nu \mu \mathcal{D}^2)^{\alpha}$ for some $\alpha>0$    depending on $\sigma$. Moreover, one can choose $\alpha>0$ such that this bound holds uniformly for $\sigma$ in any fixed compact set contained in $\sigma>1/2$. This implies the desired bound for $|\mathcal{Z}(\nu,\mu,s)|$.
\end{proof}

\begin{lemma}\label{Bbound}
Let $\nu,\mu\in \mathcal{D}^{-1}$ both non zero. Then, the series $\mathcal{B}_{\infty,\infty}(\nu,\mu;s)$ 
converges absolutely for $\Re(s)>1/2$ and the bound
\begin{align*}
|\mathcal{B}_{\infty,\infty}(\nu,\mu;s)|=
O\left(  e^{8\pi \sqrt{|\nu \mu |}}|\nu \mu|^{\Re(s)+1}  \right)
\end{align*}
holds uniformly for $s$ in any compact set contained in $\Re(s)>1/2$.
\end{lemma}
\begin{proof}
We start by writing
\begin{align*}
\mathcal{B}_{\infty,\infty}(\nu,\mu;s)=\frac{2\pi}{\covol(\mathcal{O}_K)}\sum_{\substack{c\in \mathcal{O}_K/\{\pm 1\}\\ c\neq 0}}\frac{\mathcal{S}(\nu,\mu,c)}{|c|^2}\,\J_s\left(\frac{\nu \mu}{c^2}\right).
\end{align*}
For fixed $s\in \C$ with $\mathrm{Re}(s)>1/2$, Lemma \ref{Lbound} gives 
\begin{align*}
\sum_{\substack{c\in \mathcal{O}_K/\{\pm 1\}\\ c\neq 0}}\left|\frac{\mathcal{S}(\nu,\mu,c)}{c^2}\,\J_s\left(\frac{\nu \mu}{c^2}\right)\right|&=  
 O(|\nu \mu|^{\Re(s)} \sum_{|\frac{\nu \mu}{c^2}|\leq 1 }\frac{|\mathcal{S}(\nu,\mu,c)|}{|c|^{2+2\Re(s)}} \\
&\mathrel{\phantom{=}} +
   |\nu \mu|^{\Re(s)} \sum_{\left|\frac{\nu \mu}{c^2}\right|> 
1 }\frac{|\mathcal{S}(\nu,\mu,c)|}{|c|^{2+2\Re(s)}}
e^{8\pi \sqrt{|\nu \mu/c^2 |}}\bigg). 
\end{align*}
Since
\begin{align*}
\sum_{\left|\frac{\nu \mu}{c^2}\right|> 
1 }\frac{|\mathcal{S}(\nu,\mu,c)|}{|c|^{2+2\Re(s)}} e^{8\pi \sqrt{|\nu \mu/c^2 |}}
\leq  e^{8\pi \sqrt{|\nu \mu |}}  \cdot \#\left\{c\in \mathcal{O}_K:|c|^2<|\nu \mu|\right\}
=O\left(e^{8\pi \sqrt{|\nu \mu |}}|\nu \mu|\right),
\end{align*} 
we have
\begin{align*}
\sum_{\substack{c\in \mathcal{O}_K/\{\pm 1\}\\ c\neq 0}}\left|\frac{\mathcal{S}(\nu,\mu,c)}{c^2}\J_s\left(\frac{\nu \mu}{c^2}\right)\right|
=   O\left(|\nu \mu|^{\Re(s)}  \mathcal{Z}(\nu,\mu,\Re(s))+  e^{8\pi \sqrt{|\nu \mu |}}|\nu \mu|^{\Re(s)+1}\right).
\end{align*}
This together with Lemma \ref{contLZseries} implies the absolute convergence of $\mathcal{B}_{\infty,\infty}(\nu,\mu;s)$ and the desired bound for $|\mathcal{B}_{\infty,\infty}(\nu,\mu;s)|$.  This completes the proof of this Lemma.
\end{proof}

We now give the analytic continuation of the Niebur--Poincar\'e series.

\begin{proposition}\label{ACNieburPointacare}
The Niebur--Poincar\'e series $F_{\infty,\nu}(P,s)$ has an analytic continuation to $\mathrm{Re}(s)>1/2$. Moreover, for fixed $P\in \H^3$ and $\delta>1$, the bound
$$F_{\infty,\nu}(P,s)=O_{P,\delta}\left(\max\left\{|\nu|^{\Re(s)}e^{4\pi |\nu|r},|\nu|^{\Re(s)+1}e^{\frac{4\delta \pi |\nu|}{r}}\right\}\right)$$
holds uniformly for $s$ in any compact set contained in $\Re(s)>1/2$.
\end{proposition}
\begin{proof}
By Proposition \ref{FENgeneral}, we have
\begin{align}\label{FENinfinity}
\F_{\infty,\nu}( P,s) &= rI_{s}(4\pi  |\nu| r) \sum_{u\in \mathcal{O}_K^{\times}/\{\pm 1\} }
e^{2\pi i \tr( \nu u^2z )}  +\frac{(2\pi|\nu|)^s}{s\Gamma(s)}r^{1-s}
\varphi_{\infty,\infty}(-\nu;s)\\ 
& \mathrel{\phantom{=}} + \sum_{\substack{\mu \in \Lambda_{\xi}^{\ast}\\ \mu\neq 0}}
\mathcal{B}_{\infty,\infty}(\nu,\mu;s)  rK_s(4\pi |\mu|r) e^{2\pi i \tr(\mu z)}. \nonumber
\end{align}
For $s$ in a fixed compact set in $\mathrm{Re}(s)>1/2$ we have, by Lemma \ref{Bbound} and \eqref{Kasymptotic}, the bound
\begin{eqnarray*}
 \sum_{\substack{\mu \in \Lambda_{\xi}^{\ast}\\ \mu\neq 0}}
|\mathcal{B}_{\infty,\infty}(\nu,\mu;s)  rK_s(4\pi |\mu|r)| =  O\bigg(\sqrt{r}|\nu|^{\sigma+1}\sum_{\substack{\mu\in \mathcal{D}^{-1}\\ \mu\neq 0}}
|\mu|^{\sigma+1/2}e^{8\pi \sqrt{|\nu||\mu|}}e^{-4\pi |\mu|r}\bigg),
\end{eqnarray*}
where $\sigma=\Re(s)$. The inequality
$$
e^{8\pi \sqrt{|\nu| |\mu|}}e^{-4\pi |\mu| r}\leq 
e^{\frac{4\delta |\nu| \pi }{r}}e^{-\pi (1-\delta^{-1}) |\mu| r}, 
$$
which holds for $ \delta>1$, gives
\begin{align}\label{boundNP2}
 \sum_{\substack{\mu \in \Lambda_{\xi}^{\ast}\\ \mu\neq 0}}
\left|\mathcal{B}_{\infty,\infty}(\nu,\mu;s)  rK_s(4\pi |\mu|r)\right|=O_s\bigg(\sqrt{r}|\nu|^{\sigma+1}e^{\frac{4\delta |\nu| \pi }{r}}
\sum_{\substack{\mu\in \mathcal{D}^{-1}\\ \mu\neq 0}}|\mu|^{\sigma+1/2}e^{-\pi (1-\delta^{-1}) |\mu| r}\bigg).
\end{align}
In particular, the series on the left hand side converges. This, together with the Fourier expansion \eqref{FENinfinity} and the analytic continuation of $\varphi_{\infty,\infty}(-\nu;s)$,  
give the analytic continuation of $F_{\infty,\nu}(P,s)$. Now, by the asymptotic bound \eqref{Iinfinity}, we have
\begin{equation}\label{boundconstanttermNP}
rI_{s}(4\pi  |\nu| r) =O\left(r^{\Re(s)+1}|\nu|^{\Re(s)}e^{4\pi |\nu|r}\right).
\end{equation} 
On the other hand, as mentioned in the Introduction, the function $\varphi_{\infty,\infty}(-\nu;s)$ 
can be expressed in terms of certain generalized divisors sums and it therefore
has at most polynomial growth with respect to $|\nu|$, 
uniformly for $s$ in any fixed compact set contained in $\Re(s)>0$.
This, together with \eqref{boundNP2} and \eqref{boundconstanttermNP}, gives the result on the growth 
of $|F_{\infty,\nu}(P,s)|$. This completes the proof of the Proposition.
\end{proof}

We can now state the existence  of the meromorphic continuation of $G_s(P,Q)$ together with precise information about its growth at the cusp $\infty$.

\begin{theorem}\label{GreenAC}
For fixed $P,Q\in \H^3$ with $r=r(P)>\max \{r(Q),r(Q)^{-1}\}$,
the automorphic Green's function $G_s(P,Q)$ has an analytic continuation to 
$\mathrm{Re}(s)>1/2$, $s\neq 1$, with a simple pole at $s=1$. Moreover, 
we have
\begin{equation}\label{Greenasymptotic}
\lim_{s\to 1}\bigg(G_s(P,Q)-\frac{1}{\covol(\mathcal{O}_K)}E_{\infty}(Q,s)\bigg)=-\frac{1+\log(r)}{\mathrm{vol}(X)}+o\big(1\big),
\end{equation}
as $r\to\infty$.
\end{theorem}

Note that the analyticity of $G_s(P,Q)$ for $\Re(s)>\frac{1}{2}$, $s\neq 1$, is equivalent to Sarnak's lower bound for the first ``exceptional'' discrete eigenvalue of the Laplacian on $X$ (\cite{SAR}, Theorem 3.1).

\begin{proof}
Let $P,Q\in \H^3$ and assume that $r=r(P)>\max\{r(Q),r(Q)^{-1}\}$. By Theorem \ref{prop-fourier-expansion-greens}, we have
\begin{equation*}
G_s(P,Q)=\frac{1}{\covol(\mathcal{O}_K)}\bigg(
 \frac{r^{1-s}}{s}E_{\infty}(Q,s)+2\sum_{\substack{\mu \in \mathcal{D}^{-1}\\ \mu\neq 0}} F_{\infty,-\mu}(Q,s)   \, r K_{s}(4\pi  |\mu| r)  e^{2\pi i \tr (\mu z )}\bigg).
\end{equation*}
By Proposition \ref{ACNieburPointacare}, together with the asymptotic bound \eqref{Kasymptotic}, we have
$$\sum_{\substack{\mu \in \mathcal{D}^{-1}\\ \mu\neq 0}} |F_{\infty,-\mu}(Q,s)   \, r K_{s}(4\pi  |\mu| r)|=O_{Q,\delta}\bigg(\sqrt{r} \sum_{\substack{ \mu \in \mathcal{D}^{-1}\\ \mu\neq 0}}|\mu|^{\Re(s)-\frac{1}{2}} e^{-4\pi |\mu|r}  \max\left\{e^{4\pi |\mu|r(Q)},|\mu|e^{\frac{4\delta \pi |\mu|}{r(Q)}}\right\}\bigg)$$
for any $\delta>1$, uniformly for $s$ in any compact set contained in $\Re(s)>1/2$. Choosing $\delta$ such that $r>\max\{r(Q),\delta r(Q)^{-1}\}$ we conclude that the series on the left hand side is convergent. This proves that $G_s(P,Q)$ has a meromorphic continuation to $\Re(s)>1/2$ having poles only where $E_{\infty}(P,s)$ has poles, in which case the multiplicities also agree. Since $E_{\infty}(P,s)$ admits an analytic continuation to $\Re(s)>0$, $s\neq 1$, with a simple pole at $s=1$, we conclude the same for $G_s(P,Q)$. Now, we note that the above computations also give
\begin{equation}\label{Greenasymp1}
\lim_{s\to 1}\bigg(G_s(P,Q)-\frac{r^{1-s}}{\covol(\mathcal{O}_K)s}E_{\infty}(Q,s)\bigg)=o\big(1\big),
\end{equation}
as $r\to\infty$.
A straight-forward computation using \eqref{ResEisInfinity} gives
\begin{equation}\label{Eisdifference}
\lim_{s\to 1}\bigg(\frac{r^{1-s}}{\covol(\mathcal{O}_K)s}E_{\infty}(Q,s)-\frac{1}{\covol(\mathcal{O}_K)}E_{\infty}(Q,s)\bigg)=-\frac{1+\log(r)}{\vol(X)}.
\end{equation}
Formula \eqref{Greenasymptotic} follows by combining \eqref{Greenasymp1} with \eqref{Eisdifference}. This completes the proof of the Theorem.
\end{proof}

\section{Proof of the main theorem}\label{section_proofofmainthm}

To prove our main theorem, we first introduce a building block for the class of functions 
in $\mathcal{A}$. More precisely, for $P,Q\in \mathbb{H}^3$ with $P\not=\gamma Q$ for any $\gamma\in\Gamma$, 
we define
\begin{align}\label{def_analogj}
\L(P,Q):=\lim_{s\to 1}\left(
G_s(P,Q)-\frac{1}{\covol(\mathcal{O}_K)}
\Big(E_{\infty}(Q,s)
+
E_{\infty}(P,s) 
-
\varphi_{\infty,\infty}(0;s) \Big)\right).
\end{align}
Recalling \eqref{Greenpole} and \eqref{resEis},
the above limit exists. This function can be seen as the analogue of $\log|j(\tau_1)-j(\tau_2)|$ (see Proposition 5.1 in \cite{GZ}).
The next lemma summarizes the properties of the function $\L(P,Q)$.

\begin{lemma}\label{lemm_properties_J} 
The function $\L(P,Q)$ satisfies the following properties:
\begin{enumerate}
\item[$(\L1)$] The function $\L(P,Q)$ is $\Gamma$-invariant in each variable
and can therefore be considered as a function on $X\times X$. Moreover, we have $\L(P,Q)=\L(Q,P)$.
\item[$(\L2)$] For fixed $Q\in X$, 
we have a singularity of the form
\begin{equation*}
\L(P,Q)= \frac{\nu(Q)}{2\pi} \frac{ r(Q) }{ \|P-Q\|}+O_Q(1), 
\end{equation*}
as $P \to Q$, and the function $P\mapsto \L(P,Q)$ is smooth at any point $P\in X$ with $P\neq Q$. 
\item[$(\L3)$] For $P,Q\in X$ with $P\neq Q$, we have $\Delta_{P}\L(P,Q)=0$.
\item[$(\L4)$] For fixed $Q\in X$, we have
$$\L(P,Q)=-\frac{1}{\mathrm{vol}(X)}-\frac{|\mathcal{O}_K^{\times}|}{2\covol(\mathcal{O}_K)}r^2+o(1),$$
as $r=r(P)\to \infty$.
\end{enumerate}
\end{lemma}

\begin{proof}
Properties $(\L1)$, $(\L2)$, and $(\L3)$ follow from properties $(G1)$, $(G2)$, and $(G3)$ of the Green's function $G_s(P,Q)$ together with the equality
$$\frac{1}{d(P,Q)}=\frac{r(Q)}{\|P-Q\|}+O_Q(1),$$
as $P\to Q$. 
In order to prove property $(\L4)$, we consider the Fourier expansion 
\eqref{equ_fourier_expansion_series_eis_gen}
of $E_{\infty}(P,s)$, namely the equality 
\begin{align*}
E_{\infty}(P,s)&=
\frac{|\mathcal{O}_K^{\times}|}{2}r^{1+s}
+
\varphi_{\infty,\infty}(0;s)\,r^{1-s}
+
\frac{2^{1+s}\pi^{s}}{\Gamma(s)}
\sum_{\substack{\mu \in \mathcal{D}^{-1}\\ \mu\neq 0}}|\mu|^{s}
\varphi_{\infty,\infty}(\mu;s)\,r K_{s}(4\pi |\mu|r)e^{2\pi i \tr(\mu z)},
\end{align*}
where we employed the identity $[\Gamma_{\infty}:\Gamma_{\infty}^{\prime}]=|\mathcal{O}_K^{\times}|/2$.
A straight-forward computation using \eqref{resEis} gives
\begin{align*}
\lim_{s\to 1}\varphi_{\infty,\infty}(0;s)(r^{1-s}-1)=-\frac{\covol(\mathcal{O}_K)}{\vol(X)}\log(r).
\end{align*}
Since $|\varphi_{\infty,\infty}(\mu;1)|$ is of at most polynomial growth in $|\mu|$ and $K_1(r)$ has exponential 
decay as $r\to \infty$, we therefore get
\begin{align*}
\lim_{s\to 1}\bigg(E_{\infty}(P,s)-\varphi_{\infty,\infty}(0;s)\bigg)
=\frac{|\mathcal{O}_K^{\times}|}{2}r^2-\frac{\covol(\mathcal{O}_K)}{\vol(X)}\log(r)+o(1),
\end{align*}
as $r\to \infty$. Property $(\L4)$ now follows from this together with \eqref{Greenasymptotic}.
This completes the proof of the Lemma.
\end{proof}

The function $\L(P,Q)$ is a building block for functions in $\mathcal A$. More precisely, we have the following proposition which can be seen as an analogue of \eqref{expression}.

\begin{proposition}\label{thm_block}
Let $F:\H^3\to \R\cup\{\infty\}$ be in $\mathcal{A}$, the class of functions
satisfying $(\mathcal{A}1)$--$(\mathcal{A}4)$.
Then, the limit $F(\infty):=\lim_{r\to \infty}F(P)$ exists and we have the equality
\begin{align*}
F(P)=F(\infty)+2\pi \sum_{\ell=1}^{m}
c_{\ell}\,\L(P,Q_{\ell}),
\end{align*}
for any $P\in X$ with $P\neq Q_{\ell}$, for $\ell=1,\ldots ,m$.
\end{proposition}
\begin{proof}
Let us define $\widetilde{F}(P)$, for $P\in X$ with $P\neq Q_{\ell}$, for $\ell=1,\ldots ,m$, by
$$\widetilde{F}(P)=F(P)-2\pi \sum_{\ell=1}^m c_{\ell}\,\L(P,Q_{\ell}).$$
By properties $(\mathcal{A}3)$ and $(\L3)$, we have that $\Delta \widetilde{F}(P)=0$ for $P\neq Q_{\ell}$, $\ell=1,\ldots ,m$. On the other hand, properties $(\mathcal{A}2)$ and $(\L2)$ imply that $\widetilde{F}(P)$ is locally bounded around any point in $X$. This implies that $\widetilde{F}(P)$ extends to a smooth function $\widetilde{F}:X\to \R$ satisfying $\Delta \widetilde{F}(P)=0$ everywhere. Indeed, by taking geodesic normal coordinates around any point, one can reduce the problem to the case where $\widetilde{F}(P)$ is a harmonic function with respect to the euclidean Laplacian, at least locally. The existence of the harmonic extension of $\widetilde{F}(P)$ then follows from Theorem 2.3 in \cite{ABW}.
Using $ \sum_{\ell=1}^m c_{\ell}=0$, we note that
\begin{align*}
\sum_{\ell=1}^{m}
c_{\ell}\,\L(P,Q_{\ell}) &= \sum_{\ell=1}^m c_{\ell}\lim_{s\to 1}\left(
G_s(P,Q_{\ell})-\frac{1}{\covol(\mathcal{O}_K)}
E_{\infty}(Q_{\ell},s) \right)\\
&=  \sum_{\ell=1}^m c_{\ell}\lim_{s\to 1}\left(G_s(P,Q_{\ell})-\frac{2}{\vol(X)(s^2-1)}\right)\\
&\mathrel{\phantom{=}}+\sum_{\ell=1}^m c_{\ell}\lim_{s\to 1}\left(\frac{2}{\vol(X)(s^2-1)}-\frac{1}{\covol(\mathcal{O}_K)} E_{\infty}(Q_{\ell},s)\right)\\
&= \sum_{\ell=1}^m c_{\ell}\lim_{s\to 1}\left(G_s(P,Q_{\ell})-\frac{2}{\vol(X)(s^2-1)}\right)
 +\frac{1}{\vol(X)}\sum_{\ell=1}^m c_{\ell}\log(\eta_{\infty}(Q_{\ell}) r_{\ell}).
\end{align*}
As mentioned in Section \ref{Green-subsection}, the function
$$P\mapsto \lim_{s\to 1}\left(G_s(P,Q_{\ell})-\frac{2}{\vol(X)(s^2-1)}\right)$$
is square-integrable on $X$, for fixed $Q_{\ell}$. This implies that the function
$$P\mapsto \sum_{\ell=1}^{m}c_{\ell}\,\L(P,Q_{\ell})$$
is also square-integrable. By property $(\mathcal{A}4)$, we conclude 
that $\widetilde{F}(P)$ is square-integrable over  $X$.  By Theorem 4.1.8 in \cite{EGM}, p.~140, we know that any smooth, harmonic, square-integrable function on $X$ is constant. We conclude that $\widetilde{F}(P)$ is constant. Finally, using $(\L4)$ together with $\sum_{\ell=1}^m c_{\ell}=0$, we have
\begin{align*}
\sum_{\ell=1}^mc_{\ell}\,\L(P,Q_{\ell})=o(1), 
\end{align*}
as $r\to \infty$.
We conclude that $\widetilde{F}(P)=F(\infty)$. This proves the result.
\end{proof}

We now prove our main theorem.
\begin{proof}[Proof of Theorem \ref{mainth}]
Let $F:\H^3\to \R\cup\{\infty\}$ be a function in the class $\mathcal{A}$ satisfying the properties 
$(\mathcal{A}1)$--$(\mathcal{A}4)$. By Proposition \ref{thm_block} we have
\begin{align*}
\frac{1}{\vol(X)}\int_XF(P)d\mu(P)=
F(\infty)+\frac{2\pi}{\vol(X)} \int_X\sum_{\ell=1}^{m}
c_{\ell}\,\L(P,Q_{\ell})d\mu(P).
\end{align*}
Since $\sum_{\ell}c_{\ell}=0$, we have
\begin{align*}
\int_X \sum_{\ell=1}^{m}
c_{\ell}\,\L(P,Q_{\ell})d\mu(P)&=\int_X \lim_{s\to 1}\sum_{\ell=1}^m c_{\ell}\bigg(G_s(P,Q_{\ell})-\frac{1}{\covol(\mathcal{O}_K)}E_{\infty}(Q_{\ell},s)\bigg)d\mu(P)\\
&=\int_X\sum_{\ell=1}^m c_{\ell}\lim_{s\to 1}\left(G_s(P,Q_{\ell})-\frac{2}{\vol(X)(s^2-1)}\right)d\mu(P)\\
& \mathrel{\phantom{=}} +\int_X\sum_{\ell=1}^m c_{\ell}\lim_{s\to 1}\left(\frac{2}{\vol(X)(s^2-1)}-\frac{1}{\covol(\mathcal{O}_K)} E_{\infty}(Q_{\ell},s)\right)d\mu(P).
\end{align*}
Using \eqref{G-orthogonal} and \eqref{KLF}, we obtain
\begin{align*}
\int_X \sum_{\ell=1}^{m}
c_{\ell}\,\L(P,Q_{\ell})d\mu(P)&=
 \lim_{s\to 1}\sum_{\ell=1}^m c_{\ell}\bigg(\frac{2}{s^2-1}-\frac{\vol(X)}{\covol(\mathcal{O}_K)}E_{\infty}(Q_{\ell},s)\bigg)\\
&= \sum_{\ell=1}^{m}c_{\ell} \log(\eta_{\infty}(Q_{\ell})\, r_{\ell}).
\end{align*}
This completes the proof of Theorem \ref{mainth}. 
\end{proof}

\section{Technical lemmas}\label{section-technical-lemma}

In this section we prove two lemmas that were used
in Section \ref{section-Fourier-expansions} for the computation 
of the Fourier coefficients of the Green's function and of the 
Niebur--Poincar\'e series. 

\begin{lemma}\label{Iintegral}
For $\mu,s\in \C$ with $\Re(s)>0$ and $r>\tilde{r}>0$, let
\begin{align*}
I_{\mu,s}(r,\tilde{r}):=\int\limits_{\C}\varphi_s\left(\frac{|z|^2+r^2+\tilde{r}^2}{2r\tilde{r}}\right)e^{-2\pi i \tr (\mu z )}dz
\end{align*}
with $\varphi_s(t)=(t+\sqrt{t^2-1})^{-s}(t^2-1)^{-1/2}$. Then, we have
\begin{align*}
I_{\mu,s}(r,\tilde{r})=
\begin{cases}
2\pi s^{-1}r^{1-s}\tilde{r}^{s+1}, & \text{ if }\mu= 0,\\
4\pi r\tilde{r} K_{s}(4\pi |\nu|r) I_{s}(4\pi |\nu|\tilde{r}),  & \text{ if }\mu \neq 0.
\end{cases}
\end{align*}
\end{lemma}
\begin{proof}
Using polar coordinates $z=\rho e^{i\theta}$, we get 
\begin{align*}
I_{\mu,s}(r,\tilde{r})
&=\int_{0}^{2\pi}\int_0^{\infty}
\varphi_s\left(\frac{\rho^2+r^2+\tilde{r}^2}{2r\tilde{r}}\right)
e^{-2\pi i \rho \tr( \mu e^{i\theta} )}\rho\, d\rho d\theta.
\end{align*}
Letting $t=\rho^2$ and $f(t):=(t+r^2+\tilde{r}^2)/2r\tilde{r}$, we have
\begin{align*}
I_{0,s}(r,\tilde{r})
=\pi \int_0^{\infty}\varphi_s\left(f(t)\right)dt
 =-\frac{  2\pi r\tilde{r}}{s} \left[  \left (f(t)+\sqrt{f(t)^2-1}\right)^{-s}   
\right]_{t=0}^{t=\infty}
=\frac{ 2 \pi  }{s}r^{1-s}\tilde{r}^{s+1},
\end{align*}
where we have used that $r>\tilde{r}$. This proves the first formula. For $\mu\not=0$, we write $\mu=|\mu|e^{i\alpha}$  and get
\begin{align*}
\int_{0}^{2\pi}e^{-2\pi i \rho \tr( \mu e^{i\theta})} d\theta
=\int_{0}^{2\pi}e^{-4\pi i \rho |\mu|\sin(\theta) }  d\theta
= 2\pi J_0(4\pi |\mu|\rho ),
\end{align*}
by using  formula \eqref{A1}. Replacing this in the above formula for $I_{\mu,s}(r,\tilde{r})$ and 
making the change of variables $t=\rho^2/r^2$,
we get
\begin{align*}
I_{\mu,s}(r,\tilde{r})
&=\pi r^2\int_0^{\infty}
\varphi_s\left(\frac{r}{2\tilde{r}} \left(t+1+\frac{\tilde{r}^2}{r^2}\right)\right)
J_0(4\pi |\mu| r \sqrt{t})dt.
\end{align*}
Using formula \eqref{A2}, we have
\begin{align*}
\varphi_s(b/a)=a\,
\int_0^{\infty}I_s(a u)e^{-b u}du,
\end{align*}
for $b>a>0$.
Using this identity with $a=2\tilde{r}/r$ and $b=t+1+\frac{\tilde{r}^2}{r^2}$, 
we get
\begin{align*}
I_{\mu,s}(r,\tilde{r})
&=2\pi r\tilde{r} \int_0^{\infty}I_s\left(\frac{2\tilde{r}u}{r}\right)e^{-\left(1+\frac{\tilde{r}^2}{r^2}\right) u}
\int_0^{\infty}J_0(4\pi |\mu| r \sqrt{t})e^{-t u}dtdu.
\end{align*}
Formula \eqref{A3} with $a=u$ and $b=4\pi|\mu|r$ yields
\begin{align*}
I_{\mu,s}(r,\tilde{r})
&=2\pi r\tilde{r} \int_0^{\infty}I_s\left(\frac{2\tilde{r}u}{r}\right)e^{-\left(1+\frac{\tilde{r}^2}{r^2}\right) u}
e^{-4\pi^2 |\mu|^2 r^2/u}
\frac{du}{u}.
\end{align*}
Next, we make the change of variables $t=8\pi^2 |\mu|^2 r^2/u$ 
and we get
\begin{align*}
I_{\mu,s}(r,\tilde{r})
&=2\pi r\tilde{r} \int_0^{\infty}
I_s\left(\frac{a b}{t}\right) e^{-\frac{a^2+b^2}{2 t} }e^{-t/2}
\frac{dt}{t},
\end{align*}
with $a=4\pi  |\mu| r$ and $b=4\pi  |\mu| \tilde{r}$. 
Observing that $a>b>0$ and using \eqref{A4} we conclude
\begin{align*}
I_{\mu,s}(r,\tilde{r})=4\pi r\tilde{r} K_{s}(a) I_{s}(b).
\end{align*}
This completes the proof of the Lemma.
\end{proof}

\begin{lemma}\label{Gintegral}
For $\nu,\mu,c\in \C$ with $\nu,c$ both non zero, $r>0$, and $s\in \C$ with $\Re(s)>0$, put
\begin{align*}
\I(r,\nu,\mu,c)=\int_{\C}\frac{r}{|c|^2(|z|^2+r^2)}I_s\left(\frac{4\pi |\nu|r}{|c|^2(|z|^2+r^2)}\right)e^{-2\pi i \tr\left(\nu\frac{\overline{z}}{c^2(|z|^2+r^2)}+\mu z\right)}dz.
\end{align*}
Then, we have
\begin{align*}
\I(r,\nu,\mu,c)=
\begin{cases}
\dfrac{\pi^{1+s}2^s|\nu|^s}{|c|^{2(1+s)}s\Gamma(1+s)}\,r^{1-s}, & \text{ if }\mu =0,\\[0.5cm]
\dfrac{2\pi}{|c|^2}\,rK_s(4\pi |\mu |r)\,\J_s\left(\frac{\nu \mu}{c^2}\right), & \text{ if }\mu \neq 0,
\end{cases}
\end{align*}
where $\J_s(z)$ is given in \eqref{def_Lniebur}. 
\end{lemma}
\begin{proof}
We start with the case $\mu=0$. Using polar coordinates $z=\rho e^{i\theta}$, we get
\begin{align*}
\I(r,\nu,0,c)=\frac{r}{|c|^2}\int_0^{\infty}\int_0^{2\pi}\frac{\rho}{\rho^2+r^2}I_s\left(\frac{4\pi |\nu| r}{|c|^2(\rho^2+r^2)}\right)e^{-\frac{4\pi i |\nu| \rho}{|c|^2(\rho^2+r^2)}\cos(\theta)}d\theta d\rho.
\end{align*}
Using formula \eqref{A1} and making the change of variables $\xi=\rho/r$, 
we get
$$\I(r,\nu,0,c)=\frac{2\pi r}{|c|^2}\int_0^{\infty}\frac{\xi}{\xi^2+1}I_s\left(\frac{4\pi |\nu| }{|c|^2r(\xi^2+1)}\right)J_0\left(\frac{4\pi  |\nu| \xi}{|c|^2r(\xi^2+1)}\right)d\xi. $$
Using Lemma \ref{IsJ0Integral} with $a=\frac{4\pi |\nu| }{|c|^2r}$, we get
$$\I(r,\nu,0,c)=\frac{\pi^{1+s}2^s|\nu|^s}{|c|^{2(1+s)}s\Gamma(1+s)}\,r^{1-s}.$$
This proves the first formula. For $\mu\neq 0$, we start by writing 
$$|\nu|\I\left(\frac{r}{|\mu|},\nu,\mu,c\right)=\int_{\C}\frac{|\beta|r}{|\mu z|^2+r^2}I_s\left(\frac{4\pi |\beta|r}{|\mu z|^2+r^2}\right)e^{-2\pi i \tr\left(\beta \frac{\overline{z\mu }}{(|z\mu |^2+r^2)}\right)-2\pi i \tr(\mu z)}dz$$
with $\beta=\frac{\nu \mu}{c^2}$.
Making the change of variables $\xi =\mu z$ and using polar coordinates $\xi=\rho e^{i\theta}$,
we obtain
$$|\nu|\I\left(\frac{r}{|\mu|},\nu,\mu,c\right)=\frac{|\beta| r}{|\mu|^2}\int_{0}^{\infty}\frac{\rho}{\rho^2+r^2}
I_s\left(\frac{4\pi |\beta|r}{\rho^2+r^2}\right)\int_0^{2\pi}e^{-2\pi i \left( \frac{\beta \rho e^{-i\theta}+\overline{\beta}\rho e^{i\theta} }{\rho^2+r^2}+2\rho \cos(\theta)\right)}d\theta d\rho.$$
Now, we compute
\begin{align*}
\int_0^{2\pi}e^{-2\pi i \left( \frac{\beta \rho e^{-i\theta}+\overline{\beta}\rho e^{i\theta} }{\rho^2+r^2}+2\rho \cos(\theta)\right)}d\theta
&= \int_0^{2\pi}e^{-4\pi i \frac{\rho}{\rho^2+r^2}|\beta+\rho^2+r^2|\sin(\theta)}d\theta \\
&= 2\pi J_0\left(\frac{4\pi \rho}{\rho^2+r^2}|\beta+\rho^2+r^2|\right),
\end{align*}
by formula \eqref{A1}. We conclude
\begin{align*}
|\nu|\I\left(\frac{r}{|\mu|},\nu,\mu,c\right)=\frac{2\pi |\beta| r}{|\mu|^2}\int_{0}^{\infty}\frac{\rho}{\rho^2+r^2}
I_s\left(\frac{4\pi |\beta|r}{\rho^2+r^2}\right)
J_0\left(\frac{4\pi \rho}{\rho^2+r^2}|\beta+\rho^2+r^2|\right)d\rho.
\end{align*}
Making the change of variables $t=\rho/r$, we get
\begin{align*}
|\nu|\I\left(\frac{r}{|\mu|},\nu,\mu,c\right)
&=\frac{2\pi |\beta| r}{|\mu|^2}\int_{0}^{\infty}\frac{t}{t^2+1}I_s\left(\frac{4\pi |\beta|}{r(t^2+1)}\right)
J_0\left(\frac{4\pi r t}{t^2+1}\left|\frac{\beta}{r^2}+t^2+1\right|\right)dt.
\end{align*}
Now, by formula \eqref{JmultiplicationTh} with $\lambda=\bigl|\frac{\beta}{r^2(t^2+1)}+1\bigr|$, 
$z=4\pi rt$, and $s=0$,
we get
\begin{align*}
J_0\left(\frac{4\pi rt}{t^2+1}\left|\frac{\beta}{r^2}+t^2+1\right|\right)
= \sum_{k=0}^{\infty}\frac{(-1)^k}{k!}\frac{1}{(t^2+1)^{2k}}\left(\frac{|\beta|^2}{r^4}
+\frac{\tr(\beta)}{r^2} (t^2+1)\right)^k (2\pi r t)^k J_{k}(4\pi r t).
\end{align*}
Therefore, using the binomial theorem, we have
\begin{align*}
&|\nu|\I\left(\frac{r}{|\mu |},\nu,\mu,c\right)=\\
&\frac{2\pi |\beta| r}{|\mu|^2}\sum_{k=0}^{\infty}\frac{(-1)^k(2\pi)^k}{k!}
\sum\limits_{j=0}^{k} {k\choose j }
\frac{|\beta|^{2j} \tr(\beta)^{k-j}}{r^{k+2j}} 
\int_{0}^{\infty}\frac{t^{k+1} }{(t^2+1)^{k+j+1}}
I_s\left(\frac{4\pi |\beta|}{r(t^2+1)}\right)
J_{k}(4\pi r t)
dt.
\end{align*}
Using the power expansion for $I_s(z)$ given in formula \eqref{Iseries}, we have
\begin{align*}
&\int_{0}^{\infty}\frac{t^{k+1} }{(t^2+1)^{k+j+1}}
I_s\left(\frac{4\pi |\beta|}{r(t^2+1)}\right)
J_{k}(4\pi r t)
dt
\\
&=\left(\frac{2\pi |\beta|}{r}\right)^s
\sum_{\ell=0}^{\infty}\frac{(2\pi|\beta|)^{2\ell}r^{-2\ell}}{\ell!\,\Gamma(s+\ell+1)}\int_{0}^{\infty}\frac{t^{k+1} }{(t^2+1)^{s+k+j+2\ell+1}}
J_{k}(4\pi r t)dt
\\
&=\left(\frac{2\pi |\beta|}{r}\right)^s
\sum_{\ell=0}^{\infty}\frac{(2\pi|\beta|)^{2\ell}r^{-2\ell}}{\ell!\,\Gamma(s+\ell+1)}
\frac{(2\pi r)^{s+k+j+2\ell}K_{s+j+2\ell}(4\pi r)}{\Gamma(s+k+j+2\ell+1)},
\end{align*}
by formula \eqref{A5} with $s=k$, $\mu=s+k+j+2\ell$, and $a=4\pi r$. This gives
\begin{align*}
\frac{|\nu| |\mu|^2}{2\pi |\beta| r}\I\left(\frac{r}{|\mu|},\nu,\mu,c\right)
&=\sum_{k=0}^{\infty}
\sum\limits_{j=0}^{k}\sum_{\ell=0}^{\infty}
\frac{(-1)^k(2\pi)^{2s+2k+j+4\ell}|\beta|^{s+2j+2\ell} \tr(\beta)^{k-j} K_{s+j+2\ell}(4\pi r)}
{(k-j)!j!\ell!\,\Gamma(s+\ell+1)\Gamma(s+k+j+2\ell+1)\,r^{j}}\\
&=\sum_{k=0}^{\infty}
\sum\limits_{j=0}^{\infty}\sum_{\ell=0}^{\infty}
\frac{(-1)^{k+j}(2\pi)^{2s+2k+3j+4\ell}|\beta|^{s+2j+2\ell} \tr(\beta)^{k} K_{s+j+2\ell}(4\pi r)}{k!j!\ell!\,\Gamma(s+\ell+1)\Gamma(s+k+2j+2\ell+1)\,r^{j}},
\end{align*}
using a well-known identity for the double sum over $k$ and $j$. 
Applying this identity again for the double sum over $j$ and $\ell$
and using Lemma \ref{Krecursion} with $z=4\pi r$, we get
\begin{align*}
\frac{|\nu| |\mu|^2}{2\pi |\beta| r}\I\left(\frac{r}{|\mu|},\nu,\mu,c\right)
&=\sum_{k=0}^{\infty}
\sum\limits_{j=0}^{\infty}
\frac{(-1)^{k+j}(2\pi)^{2s+2k+3j}|\beta|^{s+2j} \tr(\beta)^{k}}
{k!j!\,\Gamma(s+k+2j+1)r^{j}}
\sum_{\ell=0}^{j}
{j\choose \ell }
\frac{(-2\pi r)^{\ell} K_{s+j+\ell}(4\pi r)}
{\Gamma(s+\ell+1)}\\
&= K_s(4\pi r)\sum_{j=0}^{\infty}
\frac{(2\pi)^{2s+4j}|\beta|^{s+2j} }
{j!\,\Gamma(s+j+1)}
\sum_{k=0}^{\infty}
\frac{(-1)^{k}(2\pi)^{2k}\tr(\beta)^{k}}{k!\,\Gamma(s+k+2j+1)}.
\end{align*}
Assuming that $\tr(\beta)>0$, using formula \eqref{Jseries} and 
recalling that $\beta=\frac{\nu \mu}{c^2}$,  we therefore obtain 
\begin{align*}
\I\left(r,\nu,\mu,c\right)
&=
\frac{2\pi }{|c|^2}r K_s(4\pi |\mu|r)\sum_{j=0}^{\infty}
\frac{\bigl((2\pi)|\beta|/\sqrt{\tr(\beta)}\bigr)^{s+2j} }
{j!\,\Gamma(s+j+1)}
J_{s+2j}\left(4\pi \sqrt{\tr(\beta)}\right).
\end{align*}
Applying Lemma \eqref{Jproductlemma} with $x=4\pi \sqrt{\tr(\beta)}$ and $A=\sqrt{\beta/\tr(\beta)}$,
we get
\begin{align*}
\I\left(r,\nu,\mu,c\right)
&=
\frac{2\pi }{|c|^2}
r K_s(4\pi |\mu|r)J_s\left(4\pi\sqrt{\beta}\right)J_s\left(4\pi\sqrt{\overline{\beta}}\right)\\
 &=\frac{2\pi }{|c|^2}r
 K_s(4\pi |\mu|r)\,\J_s\left(\frac{\nu \mu}{c^2}\right).
\end{align*}
This implies the second formula in the case $\mu\neq 0, \Re\left(\frac{\nu \mu}{c^2}\right)>0$.
The case $\Re\left(\frac{\nu \mu}{c^2}\right)<0$ is completely analogous, so we omit the details. Finally, the case $\Re\left(\frac{\nu \mu}{c^2}\right)=0$ follows from any of the two other cases by taking the limit $\beta\to it, t\in \R,t\neq 0$. This completes the proof of the Lemma.
\end{proof}

\appendix

\appendixpageoff

\section{Identities involving special functions}
In this appendix we recall some identities involving special functions that are used in the paper. Most of these identities are well-known and can be found in the literature and for these we just give a reference.
For some of the less-known identities we give sketch of proofs.

We start with the well-known identities. These are
\begin{align}
\int_{0}^{2\pi}e^{-ia\sin(\theta)}d\theta&= 2\pi J_0(a),\quad a\geq 0, \label{A1}\\
\int_0^{\infty}e^{-a t}J_0(b\sqrt{t})dt&=\frac{1}{a}e^{-b^2/4a}, \quad a>0, b\in \mathbb{R}, \label{A3}\\
\int_0^{\infty}I_s\left(\frac{ab}{t}\right)e^{-\frac{t}{2}-\frac{1}{2t}(a^2+b^2)}\frac{dt}{t}
&=2K_s(a)I_s(b), \quad a>b>0, \Re(s)>-1,\label{A4} \\
\int_0^{\infty}\frac{x^{\mu-1}}{(x+1)^{\nu}}dx&=\frac{\Gamma(\mu)\Gamma(\nu-\mu)}{\Gamma(\nu)}, \quad\Re(\nu)>\Re(\mu)>0, \label{A7}
\end{align}
and
\begin{align}\label{A5}
\int_0^{\infty}\frac{t^{s+1}}{(t^2+1)^{\mu+1}}J_{s}(at)dt =\frac{(a/2)^{\mu}K_{\mu-s}(a)}{\Gamma(\mu+1)}, 
\quad a>0, -1<\Re(s)<\Re\left(2\mu+\tfrac{3}{2}\right), 
\end{align}
which can be found in \cite{GRA} (formulas 8.411-1, 6.614-1, 6.653-2, 6.565-4, and 3.194-3, respectively). 
We also have 
\begin{align}
I_s(z)&=\sum_{k=0}^{\infty}\frac{(z/2)^{s+2k}}{k!\,\Gamma(s+k+1)}, 
\quad z\in \C\setminus (-\infty,0], \label{Iseries} \\
J_s(z)&=\sum_{k=0}^{\infty}\frac{(-1)^k(z/2)^{s+2k}}{k!\,\Gamma(s+k+1)}, 
\quad z\in \C\setminus (-\infty,0], \label{Jseries}
\end{align}
and
\begin{align}
 \sum_{\ell=0}^n{n \choose \ell}(-1)^{\ell}\frac{\Gamma(\ell+b)}{ \Gamma(\ell+a)}&=
 \frac{\Gamma(n+a-b)\Gamma(b)}{\Gamma(a-b)\Gamma(n+a)},\label{A8}
\end{align}
in \begin{it}loc.~cit\end{it}.~(formulas 8.402, 8.445, and 0.160-2, respectively). For the Gauss hypergeometric series $_2F_1(a,b;c;z)$, we have the transformation property (formula 9.134-2 in \begin{it}loc.~cit\end{it}.)
\begin{align}\label{A9}
_2F_1(a,b;a-b+1;z)=(1+z)^{-a}\ _2F_1\left( \frac{a}{2},\frac{1+a}{2};a-b+1;\frac{4z}{(1+z)^2} \right).
\end{align}
Formulas 4.16-1 in \cite{MAG} and 5$\cdot$22-16 in \cite{WAT} are
\begin{align}
\int_0^{\infty}I_s(a t)e^{-bt}dt
&=\frac{a^s}{\sqrt{b^2-a^2}}(b+\sqrt{b^2-a^2})^{-s} , 
\quad \Re(s)>-1,\Re(b)>|\Re(a )|,\label{A2} \\
J_{s}(\lambda z)
&=\lambda^s\sum_{k=0}^{\infty}\frac{(-z/2)^k(\lambda^2-1)^k}{k!}J_{s+k}(z). \label{JmultiplicationTh}
\end{align}
We also have the well-known asymptotic bounds, valid uniformly for $s$ in a compact set contained in the half-plane $\Re(s)>-1/2$,
\begin{align}
|K_s(x)|&= O\left(x^{-1/2}e^{-x}\right), 
\quad\text{ for }x\to \infty,x\in \R, \label{Kasymptotic}
\end{align}
and
\begin{align}
I_s(z)&=  O\left(|z|^{\Re(s)} e^{\Re(z)}\right), \label{Iinfinity}\\
J_s(z)&=  O\left(|z|^{\Re(s)} e^{|\Im(z)|}\right), \label{Jinfinity}
\end{align}
for $z\to \infty, |\mathrm{arg}(z)|\leq \frac{1}{2}\pi - \delta$ with fixed $\delta>0$. 
The asymptotic formulas
\begin{align}\label{IJasympt}
I_s(z)\sim J_s(z)\sim \frac{(z/2)^s}{\Gamma(s+1)}, \quad\text{ for }z\to 0,
\end{align}
also hold uniformly for $s$ in a fixed compact set. 

We now give the less-known identities in the form of several lemmas.
\begin{lemma}\label{IsJ0Integral}
We have 
\begin{align*}
\int_0^{\infty}\frac{\xi}{\xi^2+1}I_s\left(\frac{a }{\xi^2+1}\right)J_0\left(\frac{a \xi}{\xi^2+1}\right)d\xi=\frac{a^s}{2^{s+1}s \Gamma(s+1)}
\end{align*}
for any $a>0$ and $s\in \C$ with $\Re(s)>0$.
\end{lemma}
\begin{proof}
Using formulas \eqref{Iseries} and \eqref{Jseries}, we have
\begin{align*}
  &\int_0^{\infty}\frac{\xi}{\xi^2+1}I_s\left(\frac{a }{\xi^2+1}\right)J_0\left(\frac{a \xi}{\xi^2+1}\right)d\xi\\
 &=\left(\frac{a}{2}\right)^s\sum_{k,j=0}^{\infty}\frac{(-1)^ja^{2k+2j}}{4^{k+j}k! j!\, \Gamma(s+k+1)\Gamma(j+1)}\int_0^{\infty}\frac{\xi^{2j+1}}{(\xi^2+1)^{2k+2j+s+1}}d\xi.
\end{align*}
By putting $x=\xi^2$, $n=k+j$, and using formula \eqref{A7}, we see that this equals
\begin{align*}
  \frac{1}{2}\left(\frac{a}{2}\right)^s\sum_{n=0}^{\infty}\frac{a^{2n}}{4^{n}n!\,\Gamma(2n+s+1)}
  \sum_{j=0}^n{n \choose j}\frac{(-1)^j\,\Gamma(s+2n-j)}{ \Gamma(s+n-j+1)}.
\end{align*}
By formula \eqref{A8} we have
\begin{align*}
\sum_{j=0}^n{n \choose j}\frac{(-1)^j\,
\Gamma(s+2n-j)}{ \Gamma(s+n-j+1)}=(-1)^n\sum_{j=0}^n{n \choose j}
\frac{(-1)^j\Gamma(s+n+j)}{ \Gamma(s+j+1)}
=
\begin{cases}
s^{-1}, & \text{ if }n=0,\\
0, & \text{ if }n\geq 1.
\end{cases}
\end{align*}
Replacing this in the previous expression gives the desired formula.
\end{proof}

\begin{lemma}\label{Krecursion}
For $s,z\in \mathbb{C}$, we have
\begin{align*}
\sum_{\ell=0}^j {j \choose \ell}\frac{(-z/2)^{\ell}K_{s+j+\ell}(z)}{  \Gamma(s+\ell+1)}
=\frac{(-z/2)^j K_s(z)}{\Gamma(s+j+1)}.
\end{align*}
\end{lemma}
\begin{proof}
This identity can be proved by induction over $j$, the case $j=0$ being obvious. 
For the inductive step one can use the identity
$$K_{s+1}(z)=K_{s-1}(z)+\frac{2s}{z}K_s(z)$$
(see formula 8.468-10 in \cite{GRA}).
We omit the details. 
\end{proof}

\begin{lemma}\label{Jproductlemma}
Assume that $A\in \C\setminus ]-\infty,0]$, $\Re(A^2)=1/2$, and $x>0$. Then
$$J_s(Ax)J_s(\overline{A}x)=\sum_{n=0}^{\infty} \frac{\left(x|A|^2/2\right)^{s+2n}}{n! \,\Gamma(s+n+1)}J_{s+2n}(x)$$
and similarly
$$I_s(Ax)I_s(\overline{A}x)=\sum_{n=0}^{\infty} \frac{\left(x|A|^2/2\right)^{s+2n}}{n! \,\Gamma(s+n+1)}I_{s+2n}(x).$$
\end{lemma}
\begin{proof}
By using formula \eqref{Jseries} and collecting the powers of $x$, we have
\begin{align*}
J_s(Ax)J_s(\overline{A}x)
= \sum_{\ell=0}^{\infty}\frac{1}{\ell!}\Bigl(-\frac{x^2}{4}\Bigr)^{\ell}\sum_{n=0}^{\ell} {\ell \choose n}\frac{A^{2n}\overline{A}^{2\ell-2n}}{\Gamma(s+n+1)\Gamma(s+\ell-n+1)}
\end{align*}
and
\begin{align*}
\sum_{n=0}^{\infty} \frac{\left(x|A|^2/2\right)^{s+2n}}{n! \,\Gamma(s+n+1)}J_{s+2n}(x)
= \sum_{\ell=0}^{\infty}\Bigl(-\frac{x^2}{4}\Bigr)^{\ell}\frac{1}{\Gamma(s+\ell+1)}\sum_{n=0}^{\left[\ell/2\right]} \frac{|A|^{4 n}}{n! (\ell-2n)!\, \Gamma(s+n+1)},
\end{align*}
where $[\ell/2]$ denotes the integral part of $\ell/2$. Putting $a=A^2$ and using that $\overline{a}=1-a$, we get
\begin{align*}
\sum_{n=0}^{\ell}{\ell \choose n} \frac{A^{2n}\overline{A}^{2\ell-2n}}{\Gamma(s+n+1)\Gamma(s+\ell-n+1)}
= \frac{(1-a)^{\ell}}{\Gamma(s+1)\Gamma(s+1+\ell)}\  _2F_1\left(-s-\ell,-\ell;s+1;\frac{a}{1-a}\right)
\end{align*}
and
\begin{align*}
 \sum_{n=0}^{\left[\ell/2\right]} \frac{|A|^{4 n}}{n! (\ell-2n)! \,\Gamma(s+n+1)}
 =\frac{1}{\ell!\,\Gamma(s+1)} \ _2F_1\left(\frac{1-\ell}{2},-\frac{\ell}{2};s+1;4a(1-a)\right).
\end{align*}
Thus, we just have to prove the 
identity 
\begin{align*}
(1-a)^{\ell}\ _2F_1\left(-s-\ell,-\ell;s+1;\frac{a}{1-a}\right)=\ _2F_1\left(\frac{1-\ell}{2},-\frac{\ell}{2};s+1;4a(1-a)\right).
\end{align*}
This follows from identity \eqref{A9} applied to $a=-\ell$, $b=-s-\ell$, and $z=a/(1-a)$. This completes the 
proof of the first formula. The proof of the second formula is completely analogous, so we omit the 
details.
\end{proof}

\end{document}